\documentclass{amsart}

\usepackage{stile2017}

\usepackage[a4paper,top=4cm,bottom=4cm,left=4cm,right=4cm]{geometry}
\usepackage{array}

\title[Long gaps for cubic and biquadratic diagonal forms]{Arbitrarily long gaps between the values of positive-definite cubic and biquadratic diagonal forms}
\author{Luca Ghidelli}
\date{\today}

\address{150 Louis-Pasteur Private, Office 608, Department of Mathematics and Statistics, University of Ottawa, Ottawa ON K1N 9A7, Canada}
\email{{luca.ghidelli@uottawa.ca}}

\newcommand{\qq}{\mfk q}

\renewcommand{\and}{\quad\ \  \text{ and }\ \ \quad}

\newcommand{\Li}{\op{Li}}
\newcommand{\FF}{\mbb{F}}
\renewcommand{\P}{{\mcl{P}}}
\newcommand{\K}{K}
\newcommand{\X}{\mfk X}
\renewcommand{\Re}{\op{\mbb Re}}
\newcommand{\xxi}{{\boldsymbol \xi}}

\newcommand{\s}{{s}}  	
\newcommand{\x}{{\mbf x}}
\newcommand{\form}{F(\x)}
\newcommand{\sform}{F}
\newcommand{\Set}{\mcl S_\sform}
\newcommand{\Gap}{\op{Gap}} 
\newcommand{\sol}{r}
\newcommand{\Sol}{{\mcl R}}
\newcommand{\erre}{r}
\newcommand{\Erre}{{\mcl R}}

\newcommand{\aphi}{\varphi}
\newcommand{\bphi}{\bar\varphi}
\newcommand{\aux}{\alpha}
\newcommand{\achi}{\chi_{aux}}
\newcommand{\pp}{{\mfk{q}}}
\newcommand{\primo}{p} 
\newcommand{\pprimo}{q} 
\newcommand{\pol}{f_{aux}}
\newcommand{\Saux}{S_{aux}}
\newcommand{\proj}{\pi}

\newcommand{\maier}{{\mcl R}} 
\newcommand{\formm}{F'(\x)}
\newcommand{\sformm}{F'}

\newcommand{\chis}{\chi_{\s,p}}
\newcommand{\chip}{\chi_{3,p}}
\newcommand{\wbchip}{\wb\chi_{3,p}}
\newcommand{\chipp}{\chi_{4,q}}
\newcommand{\chippl}{\chi_{2,q}}
\newcommand{\wbchipp}{\wb\chi_{4,q}}
\newcommand{\Hp}{H_{\sform,p}}
\newcommand{\Kp}{K_{\sform,p}}
\newcommand{\Hpp}{H_{\sform,q}}
\newcommand{\Kpp}{K_{\sform,q}}
\newcommand{\rhop}{\Re\Hp} 
\newcommand{\rhopp}{\Re\Hpp} 
\newcommand{\Pquattro}{\P_{\sform,1}}

\renewcommand{\b}{\beta}

\renewcommand{\d}{\delta}
\newcommand{\ax}{\a}
\newcommand{\bx}{\b}

\newcommand{\cxx}{C_{\sform,K}}

\newcommand{\kx}{\kappa_{\sform}}
\newcommand{\dx}{\delta}
\newcommand{\dxx}{\delta_0}
\newcommand{\gx}{\gamma_\sform}
\newcommand{\e}{\varepsilon}

\newcommand{\f}{f}
\newcommand{\Hom}{\op{Hom}}
\newcommand{\J}{\mfk J}

\newcommand{\si}{\sigma}
\newcommand{\ua}{\underline{a}}

\renewcommand{\O}{\mcl O}
\renewcommand{\a}{\alpha}
\renewcommand{\u}{\mbf u}
\newcommand{\m}{\mfk m}
\newcommand{\I}{\mcl I}
\newcommand{\p}{\mfk p}
\newcommand{\uno}{\mbf{1}}

\renewcommand{\k}{\mbf{k}}
\newcommand{\A}{\mcl{A}}
\newcommand{\B}{\mcl{B}}

\begin{document}

\begin{abstract}
For $\s=3,4$, we prove the existence of arbitrarily long sequences of consecutive integers none of which is a sum of $\s$ nonnegative $\s$-th powers.
More generally, we study the existence of gaps between the values $\leq N$ of diagonal forms of degree $\s$ in $\s$ variables with positive integer coefficients.
We find: (1) gaps of size $O\left(\frac{\sqrt {\log N}}{(\log \log N)^2}\right)$ when $\s=3$; 
(2) gaps of size $O\left(\frac{\log\log\log N}{\log\log\log\log N}\right)$  if $\s=4$ and the form, up to permutation of the variables, is not equal to $a (c_1x_1)^4+b (c_2 x_2)^4+4 a (c_3x_3)^4+4b(c_4x_4)^4$.
\end{abstract}

\maketitle

\tableofcontents

%
%

\section{Introduction}\label{sec:intro}

Let $\s\in \N_+$ and let $\form= a_1 x_1^\s + \dots+ a_\s x_\s^\s$ be a diagonal form of degree $\s$ in $\s$ variables with positive integer coefficients $a_1,\ld,a_\s\in\N_+$. 
In this article by \emph{values of $\form$} we mean the natural numbers obtained by evaluating the diagonal form at nonnegative integers $x_1,\ld,x_\s\in\N$. 
A \emph{gap} of length $K$ between these values is a sequence of consecutive nonnegative integers $n+1,\ld,n+K$ that are not values of $\form$. 
When $\s=2$ the polynomial $F(\x)$ is a multiple of a norm form and so the values of $\form$ form a set with natural density 0 in $\N$ (see Landau \cite{Landau} for the prototypical case $F(\x)=x_1^2+x_2^2$ and Odoni \cite{Odoni} for general norm forms). 
In particular if $\s=2$ there are arbitrarily long gaps between the values of $F(\x)$. 
When $\s\geq 3$ the polynomial $\form$ is irreducible over $\C$ and so it is not a norm form.  
In fact very little is known unconditionally about the distribution of the values of $\form$  if $\s\geq 3$ (see \cite{GRHdistribution} for some results conditional on GRH)  
 but it is reasonable to expect, on the basis of probabilistic models \cite{density1} \cite{density2}, that the set of values of $\form$ has positive density. 
Nevertheless, we may ask if there are arbitrarily long gaps between the values of $\form$, when $\s\geq 3$. 
In this article we give a positive answer in two cases. 
First, for all trinomial positive-definite cubic diagonal forms:

\begin{theorem}\label{thm:main:intro:3}
Let $\form$ be as above, with $\s=3$. 
Then there is a constant $\kx\great 0$ such that for all integers $N,K$ satisfying $N\great e^e$, $K\geq 2$ and $K\less\kx\frac{\sqrt{\log N}}{(\log\log N)^2}$ there exist gaps of length $K$ between the values of $\form$ less than $N$. 
\end{theorem}

Second, for almost all quadrinomial positive-definite biquadratic diagonal forms: 

\begin{theorem}\label{thm:main:intro:4}
Let $\form$ be as above, with $\s=4$, and suppose that $\form$ is not equal to $a (c_1x_1)^4+b (c_2 x_2)^4+4 a (c_3x_3)^4+4b(c_4x_4)^4$, for some $a,b,c_1,c_2,c_3,c_4\in\N_+$, up to a permutation of the variables. 
Then there is a constant $\kx\great 0$ such that for all integers $N,K$ satisfying $N\great e^{e^{e^e}}$, $K\geq 2$ and $K\less\kx\frac{\log\log\log N}{\log\log\log\log N}$ there are gaps of length at least $K$  between the values of $\form$ less than $N$.
\end{theorem}

Notice that in both theorems the upper bound on $K$ goes to infinity with $N$, 
but the growth is much faster when $\s=3$.  
We refer to \cref{rmk:3vs4} for some explanation. 
In  \cref{thm:main} we show more precisely that, for a suitable $\kappa_F>0$ and the same hypotheses, there exist at least $c(\sform,K) N$ gaps of length $K$ between the values of $\form$ less than $N$, 
where $c(\sform,K)\great 0$ is independent of $N$. 

%
%

The above theorems include the important special cases $\form= x_1^3+x_2^3+x_3^3$ and $\form=x_1^4+x_2^4+x_3^4+x_4^4$. 
The values of these forms are often studied in connection with Waring's problem  \cite{survey}, which more generally concerns the representability of natural numbers as sums of perfect powers. 
Moreover, the results of the present paper concerning these two special cases have been used in a crucial way to improve some results of Bradshaw \cite{Bradshaw} in regard to values of cubic and biquadratic theta series \cite{ghidelli:theta}. 

On the other hand \cref{thm:main:intro:4} doesn't apply to some biquadratic forms such as $\form=x_1^4 + x_2^4+4x_3^4+4x_4^4$. 
We show that these exceptions are characterized among all biquadratic diagonal forms by a local property (see \cref{thm:exceptional:iff}). 
This is further discussed in \cref{rmk:5andexceptional}.

We now compare the above results with the literature. 
When $\s=2$ Richards \cite{squares} proved, with an ingenious elementary proof, that there are gaps of length at least $\gamma_\sform\log N$ between the values of $\form$, for some constant $\gamma_\sform\great 0$. 
It is an important open-problem to estimate sharply the order of growth of the gaps between the values of $\form=x_1^2+x_2^2$. 
However when $\s\geq 3$ our knowledge is even weaker. 
For example, if $\form=x_1^3+x_2^3+x_3^3$ we only know by an elementary greedy argument \cite{gaps} that for $N$ large enough there are no gaps of size greater than $3^{19/9} N^{8/27}(1+o(1))$, among the values of $\form$ less than $N$. 
On the other hand, working out the predictions of the probabilistic models, we should expect the existence of gaps of length as large as $O(\log N/\log \log N)$, for all $\s\geq 3$. 

In the following section we expose our strategy towards the proofs of \cref{thm:main:intro:3,thm:main:intro:4}. 
As it will be clear, the same method can be used to prove the existence of arbitrarily long gaps between the values of other polynomials, provided they satisfy a certain local property (see ``Step 2'' below). 
Following a suggestion of Wooley, we are going to treat in a future publication the case of non-homogeneous diagonal forms such as $x_1^2+x_2^3+x_3^7+x_4^{42}$.  


%
%

\section{Detecting the existence of long gaps - the method}\label{sec:plan}

Let $\s,\form$ be as in \cref{thm:main:intro:3,thm:main:intro:4}.  
Let $\Set\subseteq\N$ be the set of values of $\form$ and for all $n\in\N$ let $\erre_\sform(n):=\#\{\x\in\N^\s:\ \form = n\}$ be the number of representations of $n$ as a value of $\form$. 
Moreover, for all $M\in\N_+$ and $m\in\Z$ let $\sol_\sform(m,M)$ denote the number of solutions $\x\in(\Z/M\Z)^\s$ to the congruence $\form\equiv m\pmod M$. 
Our strategy to find gaps between the values of $\form$ consists of three parts:
\begin{description}
\item[Step 1] 
Estimate $\sol_\sform(m,p)$ for prime numbers $p$, with special attention to the case $m=0$. 
In particular we find a set $\P_\sform$ of primes and positive real numbers $\{\epsilon_p\}_{p\in\P_\sform}$ with the following properties: 
$\sol_\sform(0,p)\leq p^{\s-1}(1-\epsilon_p)$ for all $p\in\P_\sform$, and $\sum_{p\in\P_\sform}\epsilon_p = +\infty$.  
\item[Step 2] 
Show that for every $\epsilon\great 0$ and $K\in\N_+$ there are $m,M\in\N$ with $0\leq m\less M-K$ such that $\sol_\sform(m+k,M)\less \epsilon M^{\s-1}$ for all $k=1,\ld,K$.
\item[Step 3] 
Form the intersection of $\Set$ with a set of the form \[\maier=\{m+k + (h-1)M:\ 1\leq k\leq K,\  1\leq h\leq H\}.\] 
If $M$ and $m$ are obtained from Step 2 with $\epsilon\less \frac 1 K$, and $H$ is suitably chosen, we find that the cardinality of the intersection is strictly less than $H$. 
This implies that $m+(h_0-1)M+ [1,K]$ is a gap between the values of $\form$, for some $h_0\leq H$.
\end{description}

The underlying idea is the following: suppose that the number of solutions to the congruence $\form\equiv m \pmod M$ is significantly smaller than the ``expected'' number $M^{\s-1}$; then a number of the form $m+ (h-1)M$ has a low chance to be a value of $\form$, if $h$ is randomly chosen. In other words, these numbers are likely to be in a gap of $\form$.
To make this observation rigorous in Step 3, we require that the form $\form$ is positive-definite.

The first step constitutes the bulk of this article, and occupies all the sections from 3 to 7. 
Steps 2 and 3 are performed in \cref{sec:main}, together with the derivations of the quantitative estimates announced in \cref{sec:intro}. 
We now give more details about the strategy outlined above, in the case of biquadratic diagonal forms. 
The case of cubic forms is analogous: it is only slightly more delicate in Step 2, and overall considerably easier in Step 1. 
See also \cref{rmk:easy} for some variants of our proof.

\subsection{Step 1}\label{sec:plan:1}
 
Let $\s=4$, then fix $\form$ as in \cref{thm:main:intro:4}, and let $\Sigma_\sform$ be the set of primes that divide some coefficient of $\form$.
The outcome of Step 1 is the following.
\begin{proposition}\label{prop:step1}
For all $m\in\Z$ and all prime $p\equiv 1 \pmod 4$ with $p\not\in\Sigma_\sform$, we have 
\[
\sol_\sform(m,p)\leq p^3(1+81 p^{-3/2}).
\]
Moreover for all $\beta\in (0,1)$ there is a set of primes $\P_\sform$ with positive relative density $\delta\great 0$ such that for all $p\in\P_\sform$ we have $p\equiv 1\pmod 4$ and:
\begin{equation}\label{eq:step1:0}
\sol_\sform(0,p)\leq p^3\left(1-\beta p^{-1}\right).
\end{equation}
\end{proposition}
The first upper estimate for $\sol_\sform(m,p)$ is a consequence of the Deligne-Weil bounds \cite[Chapter 4.5]{Serre} (see \cref{erre:1} below). 
The second result for $\sol_\sform(0,p)$ comes from an exact formula of the form
\begin{equation}\label{eq:step1:exact}
\sol_\sform(0,p)= p^3 + p (p-1) (2\Re \Hp + \Kp)
\end{equation}
which is established in \cref{sec:char,sec:diagonal} using the theory of cyclotomy, more precisely with Gauss and Jacobi sums \cite{book:GJS} \cite[Sec. 8]{Ireland}. 
Here $\Hp$ and $\Kp$ denote explicit character sums modulo $p$, where $p\equiv 1 \pmod 4$ is prime and $p\not\in\Sigma_\sform$. 
Moreover, $\Hp$ is a complex number of absolute value 1 well-defined up to conjugation, and $\Kp$ is an integer satisfying $-7\leq \Kp\leq 19$.
The formula \eqref{eq:step1:exact} is related to the Sato-Tate distribution \cite[Chapter 8]{Serre} of the affine scheme associated to $\form$: the continuous part of the Sato-Tate distribution corresponds to $\Hp$, and the discrete part to $\Kp$.

By a theorem of Weil, we are able to interpret $\Hp$ as a Hecke character of infinite order and absolute value 1. 
Using the theory of Hecke L-functions, we prove in \cref{sec:equidistribution} that $\Hp$ equidistributes on the unit circle (up to conjugation) as $p\to\infty$. In particular, for all $\beta\in(0,1)$ we have $2\Re \Hp\less -1-\beta$ for a positive proportion of the primes.
 
On the other hand, in \cref{sec:kummer} we relate $\Kp$ to the Kummer extension $L/K$, where $K:=\Q(i)$, and $L=K(\sqrt[4]{a_1},\ld,\sqrt[4]{a_4},\sqrt[4]{-1})$ is generated by the fourth roots of -1 and of the coefficients of $\form$. 
By Chebotarev's theorem and Kummer's theory, we are able to compute the possible values of $\Kp$ explicitly from the characters of $\op{Gal}(L/K)$, which is a finite abelian group of order at most 512. 
In particular we can show that $\Kp\leq 1$ for a positive proportion of the primes, if $\form\neq\  a (c_1x_1)^4+b (c_2 x_2)^4+4 a (c_3x_3)^4+4b(c_4x_4)^4$ up to a permutation of the variables. 
In fact, this hypothesis on $\form$ is necessary to have $\Kp\leq 1$, as we show in \cref{sec:exceptional} by an elementary argument.

\subsection{Step~2}\label{sec:plan:2}

In order to construct $M$ and $m$, we start by choosing suitable disjoint finite subsets $\P_1,\ld,\P_K$ of $\P_\sform$ and we form their union $\P:=\P_1\cup\ld\cup\P_K$. Then, we let $M$ be the (squarefree) product of all $p\in\P$, and we take $m$ so that $m+k\equiv 0\pmod p$ for all $k\leq K$ and all $p\in\P_k$. 
In this way, by the Chinese Remainder Theorem and the estimates of Step 1, we have
\[ \begin{aligned}
 \log\left(\frac {\sol_\sform(m+k,M)}{M^3}   \right) %
   &\leq    \log\left(  %
                       \prod_{p\in\P\setminus\P_k} \left(1+\frac{81}{p^{3/2}}\right) %
                      \prod_{p\in\P_k} \left(1-\frac{\beta}{p}\right) %
                \right)   \\
    & \leq  81 \sum_{p\in\P\setminus\P_k}\frac {1} {p^{3/2}}        -      \beta \sum_{p\in\P_k} \frac {1} {p} \\
\end{aligned}  \]
for all $1\leq k\leq K$. 
We notice that the series $\sum_{\text{$p$}} p^{-3/2}$ ranging over all primes is bounded above by an absolute constant $C_1$. 
On the other hand, since $\P_\sform$ has positive density, we have that $\sum_{p\in\P_\sform} p^{-1}$ diverges, and therefore it is possible to choose $\P_1,\ld,\P_K$ so that $\sol_\sform(m+k, M)\leq \epsilon M^3$, for all $k\leq K$ and for any given $\epsilon\great 0$. 

\subsection{Step~3}\label{sec:plan:3}

The conclusion is now obtained by a simple double-counting technique that is sometimes known as the Maier matrix method \cite{Maiermatrix}. 
Fix $K\in\N_+$ and $0\less\epsilon\less \frac 1 K$, and construct $M,m$ as in Step 2, with $0\leq m\less M-K$.  
Let $\maier=\{m+k + (h-1)M:\ 1\leq k\leq K,\  1\leq h\leq M^3\}$.  
Since $\form\less M^4$ implies $x_1,x_2,x_3,x_4\in\{0,\ld,M-1\}$, we have:
\[
\begin{aligned}
 \#\Set \cap \maier &\leq \sum_{k=1}^K\sum_{h=1}^{M^3}\#\{\x\in\N^4:\ \form= m+k+(h-1) M\} \\
 & \leq
 \sum_{k=1}^K\#\{\x\in(\Z/M\Z)^4:\ \form\equiv m+k\pmod M\},
\end {aligned}
\]
which is equal to $\sum_{k=1}^K \sol_\sform(m+k,M)$, and so it is at most $K\epsilon M^3\less M^3$ by Step 2. 
On the other hand, suppose by contradiction that for all $h\leq M^3$ the interval $m+[1,K] +(h-1)M$ contains a value of $\form$. 
Then $\# \Set \cap \maier$ contains at least $M^3$ elements, and this is a contradiction.

\begin{remark}
A modification of Steps 1 and 2 proves the existence of residue classes
$m\pmod M$ that satisfy $\sol_\sform(m,M)\great c M^{\s-1}$ for arbitrarily large $c\great 0$. 
This can be used to show (see \cite[Chapter IV.1]{Hooley}) that for any given $A\great 0$ there exists $n\in\N_+$ such that the equation $\form=n$ has at least $A$ solutions $\x\in\N^\s$.
\end{remark}

\begin{remark}\label{rmk:3vs4}
 If $s=3$ we have an analog of \eqref{eq:step1:0} of the form
\begin{equation*}
\sol_\sform(0,p)\leq p^2\left(1-\beta p^{-1/2}\right),
\end{equation*}
so Step 1 is fulfilled with $\epsilon_p\asymp p^{-1/2}$. 
Then the series $\sum_{p\in\mcl P_F} \epsilon_p \asymp \sum_{p\in\mcl P_F} p^{-1/2}$ diverges to infinity much faster than the series $\sum_{p\in\mcl P_F} p^{-1}$ which appears in Step 2 above, in the case $s=4$.  
This is the technical reason that explains why the estimate on $K$ in our main result \cref{thm:main:intro:3} for cubic forms is much better than the one for biquadratic forms in \cref{thm:main:intro:4}.
\end{remark}

\begin{remark}\label{rmk:5andexceptional}
 When $s\geq 5$, it is well known \cite{density1} that
\begin{equation*}
\sol_\sform(m,q) =  q^{s-1}\left(1+ O(q^{-3/2})\right)
\end{equation*}
for every power of a prime $q=p^\nu$ and every residue class $m\bmod q$. 
Reasoning as in Step 2, since the series $\sum_{q\in\N_+} q^{-3/2}$ converges, we see that there exist positive constants $c_0,c_1$ such that 
\begin{equation*}
 c_0 M^{s-1} \leq r_F(m,M) \leq c_1 M^{s-1}
\end{equation*}
for all $M\in\N_+$ and all $m\bmod M$. 
This explains why our approach does not yield arbitrarily long gaps between the values of diagonal forms in 5 or more variables. 
Step 2 also fails when $s=4$ and 
\begin{equation*}
  \form = a (c_1x_1)^4+b (c_2 x_2)^4+4 a (c_3x_3)^4+4b(c_4x_4)^4 
\end{equation*}
for some $a,b,c_1,c_2,c_3,c_4\in\N_+$, because any such form satisfies $M^{3} \leq r_F(0,M)$ for all odd squarefree moduli $M$ (see \cref{sec:exceptional}). 
Taking into account higher powers of primes and the residue classes other than zero, it is in fact possible to prove that $c M^{3} \leq r_F(m,M)$ for all $m,M$ with a constant $c=c(F)>0$. 
\end{remark}

%
%

%
%

\section{Multiplicative characters and diagonal congruences}\label{sec:char}

\subsection{Characters and character sums}\label{sec:char:def}

If $\FF$ is a field we denote by $\FF^\times:=\FF\setminus\{0\}$ the multiplicative group of its nonzero elements.
A multiplicative character of $\FF$ is by definition a group homomorphism $\chi\in \Hom(\FF^\times,\C^\times)$. 
We denote by $\uno$ the trivial character, i.e. the one satisfying $\uno(t)=1$ for all $t\in\FF^\times$.
If $\chi$ is a nontrivial multiplicative character of $\FF$, it is customary to declare $\chi(0)=0$, thus extending $\chi$ to a map $\chi:\FF\to \C$. 
Given nontrivial multiplicative characters $\chi_1,\ld,\chi_\ell$ of a finite field $\FF$ we consider the generalized Jacobi sum
\begin{equation}\label{def:Jacobi}
J(\chi_1,\ld,\chi_\ell) :=\! \!\sum_{\substack{t_1,\ld,t_\ell\in\FF\\ t_1+\dots+t_\ell = 1}}\prod_{i=1}^\ell \chi_i(t_i).
\end{equation}
and we let $J_0(\chi_1,\ld,\chi_\ell)$ be defined analogously, but with the sum performed over the $\ell$-tuples satisfying $t_1+\dots+t_\ell = 0$.
If $\#\FF=p$ is a prime number, then the finite field $\FF$ is canonically isomorphic to $\FF_p:=\Z/p\Z$.
For every $\s\in\N_+$ we define 
\[
\X^{(\s)}_{p}:=\{  \chi\in \Hom(\FF_p^\times,\C^\times):\ \chi^\s =\uno \text{ and } \chi\neq\uno\}
\]
to be the set of the nontrivial multiplicative characters of $\FF_p$ with order dividing $\s$.
We observe that $\FF_p^\times$ is a cyclic group of order $p-1$, so every multiplicative character $\chi$ of $\FF_p$ is determined by its value at the multiplicative generators modulo $p$, and $\#\X^{(\s)}_p = \op{gcd}(\s,p-1) - 1$.
Since the complex exponential function is periodic with period $2\pi i$, the map $x\mapsto e^{\frac{2\pi i x}{p}}$ gives a well-defined additive character of $\FF_p$.
If $\chi$ is a multiplicative character of $\FF_p$, its associated Gauss sum is
\[
G(\chi) := \sum_{t\in\FF_p} \chi(t)e^{\frac{2\pi i t}{p}}.
\]

\subsection{Cubic and biquadratic power residue characters}\label{sec:char:power}

For  $\s\in\N_+$, let $\zeta_\s:=e^{\frac{2\pi i}{\s}}$ and let $\mu_\s:=\{\zeta_\s^i:\ 0\leq i \less \s\}\subseteq \C$.
Let $\K$ be a number field containing $\mu_\s$, let $\O_K$ be its ring of integers, and let $\p$ be a prime ideal of $\O_K$ not dividing $\s$. The discriminant of $X^\s -1$ is divisible only by the primes dividing $\s$, and so the elements $\zeta_\s^i\in\mu_\s$ are pairwise incongruent modulo $\p$. Thus $\mu_\s \bmod \p$ has cardinality $\s$, and is the complete set of $\s$-th roots of unity in the residue field $\O_K/\p$. This implies that $\s\divides N\p-1$, where $N\p:=\#(\O_K/\p)$ is the norm of $\p$.
From this we conclude that for every $a\in \O_K$ with $a\not\in\p$ there is a unique $\chi_{\s,\p}(a)\in\mu_\s$, also denoted by $\left(\frac {a} {\p}\right)_\s$ (see \cref{def:symbol} below), such that
\[
\chi_{\s,\p}(a) \equiv a^{\frac{N\p-1}{\s}} \pmod \p.
\]
The multiplicative character $\chi_{\s,\p}$ of $\O_K/\p$ is the $\s$-th power residue character modulo $\p$. 
Fix now $\s\in\{3,4\}$. 
The ring $\Z[\zeta_\s]$ is an Euclidean domain (it is the ring of Eisenstein integers for $\s=3$ and the ring of Gaussian integers for $\s=4$), and coincides with the ring of integers of the quadratic number field $\Q(\zeta_\s)$.
If $p$ is a prime number satisfying $p\equiv 1\pmod s$, then it splits in $\Z[\zeta_\s]$.
We choose an arbitrary prime $\p$ above $p$, so that $p\Z[\zeta_\s]=\p\,\wb\p$, and we define $\chis:=\chi_{\s,\p}$.
We notice that $N\p=p$, and so we may, as we will, consider $\chis$ as a multiplicative character of $\FF_p$. 
It is easy to see that the order of $\chis$ is exactly $\s$.
Finally, we let for brevity
\begin{equation}\label{eq:pi}
\pi_{\s,p}:=J(\chis,\chis).
\end{equation}
It is well-known \cite[Sec.9.4, Lemma 1, Proposition 9.9.4]{Ireland} that $\p=(\pi_{\s,p})$ and that $p=\pi_{\s,p}\wb{\pi}_{\s,p}$.

\subsection{The number of solutions of diagonal congruences}\label{sec:char:diagonal}

Let $s,k\in \N_+$ and fix a diagonal form $\form= a_1 x_1^k + \ld + a_\s x_{\s}^k$ of degree $k$ in $\s$ variables, with nonzero integer coefficients $a_1,\ld,a_\s\in\Z\setminus\{0\}$.
Let $\Sigma_\sform$ be the (finite) set of primes dividing $a_1\cdots a_\s$.
For all $M\in\N_+$ and $m\in\Z$ we define $\sol_\sform(m,M):=\#\Sol_\sform(m,M)$, where
\[
\Sol_\sform(m,M) :=\{\x\in(\Z/M\Z)^\s :\ \form \equiv m \pmod M\}.
\]
In other words, we count the solutions of the congruence $\form\equiv m \bmod M$.
A classical application of the Chinese Remainder Theorem shows that the function $\sol_\sform(m,M)$ is multiplicative in its second variable.  This allows us to reduce the computation of $\sol_\sform(m,M)$ to the case of prime moduli, if $M$ is squarefree. 
When $p$ is prime and $m$ is not divisible by $p$ we content ourselves with classical estimates for $\sol_\sform(m,p)$.
On the other hand, for $\sol_\sform(0,p)$ we will use an explicit computation in terms of modified Jacobi sums, which in turn can be computed via Gauss sums.
\begin{lemma}\label{erre:multiplicative}
Let $m\in\Z$ and $M=\prod_{i=1}^\ell p_i$ for distinct primes $p_1,\ld,p_\ell$. 
Then
\begin{equation}\label{eq:TCR}
\sol_\sform(m,M) = \prod_{i=1}^\ell\sol_\sform(m,p_i).
\end{equation}
\end{lemma}


\begin{proof}
We have \eqref{eq:TCR} because the Chinese Remainder Theorem provides a bijection
\[
\psi: \ \Sol_\sform(m,M) \to \Sol_\sform(m,p_1)\times\dots\times\Sol_\sform(m,p_\ell),
\]
sending an $\s$-tuple $(x_1,\ld,x_{\s})\in(\Z/M\Z)^\s$ to the sequence of $\s$-tuples $(x_1^{(i)},\ld,x_{\s}^{(i)})\in(\Z/p_i\Z)^\s$ with $1\leq i\leq \ell$ obtained by reducing modulo $p_i$ each component. 
\end{proof}

\begin{proposition}\label{erre:1}
Let $p$ be a prime number with $p\not\in\Sigma_\sform$, and let $m\in\Z$ witht $p\ndivides m$. Then
\[
\abs{\sol_\sform(m,p) - p^{\s-1}} \leq (k-1)^\s p^{\frac{\s - 1} 2},
\]
\end{proposition} 
\begin{proof}
This follows from the case $b\neq 0 $ of \cite[Sec. 8.7, Theorem 5]{Ireland}, since we have  $p^{\frac{\s}{2}-1}\leq p^{\frac{\s-1}{2}}$ and $\#\X_p^{(k)}\leq k-1$.
\end{proof}

%
%

\begin{proposition}[{\cite[Sec. 8.7, Theorem 5]{Ireland}}]\label{erre:Jacobi}
Let $p$ be a prime number with $p\not\in\Sigma_\sform$. Then
\[
\sol_\sform(0,p) = p^{\s-1}+ \sum_{\chi_1,\ld,\chi_{\s}}\wb\chi_1(a_1)\cdots\wb\chi_{\s}(a_\s) J_0(\chi_1,\ld,\chi_{\s})
\]
where the sum ranges over the $\s$-tuples of characters $\chi_i\in\X_p^{(k)}$ that satisfy $\chi_1\cdots\chi_{\s}=\uno$, and where $\wb\chi_i$ denotes the complex conjugate of $\chi_i$.
\end{proposition} 

\begin{proposition}[{\cite[Sec. 8.5, Prop 8.5.1 \& Cor. 1]{Ireland}}]\label{Jacobi}
Let $p$ be a prime number and let $\chi_1,\ld,\chi_\ell$ be nontrivial multiplicative characters of $\FF_p$ such that $\chi_1\cdots\chi_\ell=\uno$. Then
\begin{align}
 J(\chi_1,\ld,\chi_{\ell-1}) &=  \frac{\chi_\ell(-1)}{p} G(\chi_1)\cdots G(\chi_\ell) ; 	\label{3.1}\\
 J_0(\chi_1,\ld,\chi_{\ell}) &=\frac{p-1}{p} G(\chi_1)\cdots G(\chi_\ell). 			\label{3.2}
\end{align}
\end{proposition}

%
%

\section{The zero residue class in the cubic and biquadratic cases}\label{sec:diagonal}

\subsection{Evaluation of the Jacobi sums}\label{sec:diagonal:Jacobi}

We now specialize to the case $\s=k\in\{3,4\}$. We first compute the modified Jacobi sums appearing in \cref{erre:Jacobi}, using the notation introduced in \cref{sec:char:power}. We will then get an explicit formula for $\sol_\sform(0,p)$. In the next sections we will use it to deduce good upper bounds on $\sol_\sform(0,p)$ for special choices of $p$. 
Recall from \cref{sec:char:power} and \eqref{eq:pi} the definition of $\chis$ and $\pi_{\s,p}$ for $\s\in\{3,4\}$ and $p\equiv 1\bmod \s$.

\begin{lemma}\label{explicit:Jacobi}
Let $p$ be a prime number with $p \equiv 1\pmod 3$. Then
\begin{align}
J_0(\chip,\chip,\chip) &= (p-1)\pi_{3,p}.			\label{4.1}
\end{align}
Analogously, let $q$ be a prime number with $q\equiv 1 \pmod 4$. Then 
\begin{align}
J_0(\chipp,\chipp,\chipp,\chipp) &=(q-1)\pi_{4,q}^2;		\label{4.2}\\
J_0(\chipp,\chipp,\chipp^3,\chipp^3) &=q(q-1);				\label{4.3}\\
J_0(\chipp^2,\chipp^2,\chipp,\chipp^3) &=q(q-1)\chipp(-1);	\label{4.4}\\
J_0(\chipp^2,\chipp^2,\chipp^2,\chipp^2) &=q(q-1).		\label{4.5}
\end{align}
\end{lemma}

\begin{proof}
Since $\chip(-1)=\chip((-1)^3)=1$, equation \eqref{4.1} is a direct consequence of \eqref{3.1} and \eqref{3.2} applied to the triple of characters $(\chip,\chip,\chip)$.
It is immediate to see from the definitions that $J_0(\chipp,\chipp^3)\chipp(-1)=J_0(\chipp^2,\chipp^2)=q-1$ and that $\chipp^2(-1)=\chipp((-1)^2)=1$. 
Then \eqref{3.2} applied to the tuples of characters $(\chipp,\chipp^3)$, $(\chipp^2,\chipp^2)$ and \eqref{3.1} applied to  $(\chipp,\chipp,\chipp^2)$ give respectively
\begin{align}
&G(\chipp)G(\chipp^3) = \chipp(-1) q;				\label{4.6}\\
&G(\chipp^2)G(\chipp^2) = q ;					\label{4.7}\\
&G(\chipp)G(\chipp)G(\chipp^2) = q\,\pi_{4,q}.		\label{4.8}
\end{align}
Combining \eqref{4.7} and \eqref{4.8} we get
\begin{equation}
G(\chipp)^4 = q\,\pi_{4,q}^2.					\label{4.9}
\end{equation}
Now, \eqref{4.2}-\eqref{4.5} follow at once from  \eqref{4.6}-\eqref{4.9} and \eqref{3.2}.
\end{proof}

\subsection{Cubic and biquadratic diagonal congruences}\label{sec:diagonal:sol}

The required estimate in the case of cubic diagonal forms in 3 variables is readily obtained.
\begin{proposition}\label{erre:modp:3}
Let $\form= a_1x_1^3+a_2x_2^3 + a_3x_3^3$ with $a_1,a_2,a_3\in\Z\setminus\{0\}$ and let $p$ be a prime number with $p\not \in\Sigma_\sform$ and $p \equiv 1\pmod 3$. 
Then
\begin{equation}
\sol_\sform(0,p) =  p^2+2\rhop(p\sqrt p - \sqrt p),
\end{equation}
with $\Hp:=\wbchip(a_1a_2a_3)\pi_{3,p}/ \sqrt p$.
\end{proposition}
\begin{proof}
There are only two nontrivial cubic characters of $\FF_p$: 
$\X_p^{(3)}=\{\chip,\wbchip\}$. 
Notice that $\chip^{-1}=\chip^2=\wbchip$.
Therefore by \cref{erre:Jacobi}, \cref{explicit:Jacobi} and the multiplicativity of characters, we get 
\[
\begin{aligned}
\sol_\sform(0,p) &= p^2+\wbchip(a_1a_2a_3) (p-1)\pi_{3,p}+\chip(a_1a_2a_3)(p-1)\wb\pi_{3,p}\\
                  &=p^2+2(p-1)\Re(\wbchip(a_1a_2a_3)\pi_{3,p}).
\end{aligned}
\]
\end{proof}

The case of biquadratic diagonal forms comes with some extra complication, so we introduce some notation.
Let $q$ be a prime number with $q \equiv 1\pmod 4$ and let $\ua=(a_1,a_2,a_3,a_4)\in\Z^4$ with $q\ndivides a_1a_2a_3a_4$.
We denote by $\chipp(\ua)$ the quadruple 
\[
\chipp(\ua):=(\chipp(a_1),\chipp(a_2),\chipp(a_3),\chipp(a_4))\in\mu_4^4,
\]
where $\mu_4=\{1,-1,i,-i\}$. We say that two quadruples $\u_1,\u_2\in\mu_4^4$ are equivalent if $\u_2$ can be obtained from $\u_1$ by performing some or all of the following operations: (1) permutation of the components; (2) componentwise multiplication by an element of $\mu_4$; (3) componentwise complex conjugation.
The quotient $\mu_4^4/\!\sim$ obtained by this equivalence relation has 8 elements, displayed in \cref{table}. 
For later reference, we label these 8 elements with the names $U_1,\ld,U_8$.
We denote the equivalence class of an element $\u\in\mu^4_4$ by $[\u]$.

\begin{table}
\centering
\begin{tabular}{cccccc}
\toprule
 &	 $[\chipp(\ua)]\in \mu_4^4/\!\sim$ 	& 	\quad $b_{\sform,q}$\ \quad	&	\quad $c_{\sform,q}$\ \quad\quad  	& 	$b_{\sform,q}+c_{\sform,q}$ 	& 	$b_{\sform,q}-c_{\sform,q}$	\\
\midrule
$U_1$ 	&	 [(1, 1, 1, 1)] 	& 7	& 12	& 19	& -5 	 	\\
$U_2$ 	&	 [(1, 1, 1,-1)] 	& -5	& 0	& -5	& -5	 	\\
$U_3$ 	&	 [(1, 1, 1, i\,)]	& -1	& -6	& -7	& 5	 	\\
$U_4$ 	&	 [(1, 1,-1,-1)]	& 7	& -4	& 3	& 11	 	\\
$U_5$ 	&	 [(1, 1,-1, i\,)] 	& -1	& 2	& 1	& -3	 	\\
$U_6$ 	&	 [(1, 1, i, i\,)]	& 3	& 4	& 7	& -1	 	\\
$U_7$ 	&	 [(1, 1, i,-i\,)]  	& -1	& 0	& -1	& -1	 	\\
$U_8$ 	&	 [(1,-1, i,-i\,)] 	& 3	& -4	& -1	& 7	 \\	
\bottomrule
\end{tabular}
\vspace{3pt}
\caption{Table displaying the quantities appearing in \cref{erre:modp:4}.} \label{table}
\end{table}

\begin{proposition}\label{erre:modp:4}
Let $\form= a_1x_1^4 + a_2x_2^4  + a_3x_3^4 + a_4x_4^4$ with $\ua\in(\Z\setminus\{0\})^4$ as above. 
Let $q$ be a prime number with $q\not\in\Sigma_\sform$ and $q \equiv 1\pmod 4$.
Then
\begin{equation}\label{eq:modp:4}
\sol_\sform(0,q) = q^3+(2\rhopp + \Kpp)\, q(q-1),
\end{equation}
where 
\begin{align*}
\Hpp&:=\wbchipp(a_1a_2a_3a_4)\pi^2_{4,q}/q,\\
\Kpp&:=b_{\sform,q}+\chipp(-1)c_{\sform,q},
\end{align*} 
and $b_{\sform,q}, c_{\sform,q}\in\Z$ depend on $[\chipp(\ua)]$ as indicated  in \cref{table}.
\end{proposition}
\begin{proof}
We have $\X_q^{(4)}=\{\chipp,\chipp^2,\chipp^3\}$, i.e. there are only three nontrivial biquadratic characters of $\FF_q$. 
Thus \cref{erre:Jacobi} and \cref{explicit:Jacobi} give 
\[
\sol_\sform(0,q) = q^3 + q(q-1)(b_{\sform,q}+\chipp(-1)c_{\sform,q}) + (q-1)d_{\sform,q},
\]
where
\[
\begin{aligned}
b_{\sform,q}&=\chipp^2(a_1a_2a_3a_4) + \frac 1 4 \sum_{\sigma\in\mfk S_4}\chipp(a_{\sigma(1)}a_{\sigma(2)})\chipp^3(a_{\sigma(3)}a_{\sigma(4)});\\
c_{\sform,q}&=\frac 1 2 \sum_{\sigma\in\mfk S_4}\chipp^2(a_{\sigma(1)})\chipp^2(a_{\sigma(2)})\chipp(a_{\sigma(3)})\chipp^3(a_{\sigma(4)});\\
d_{\sform,q}&= \wbchipp(a_1a_2a_3a_4)\pi_{4,q}^2 + \chipp(a_1a_2a_3a_4)\wb\pi_{4,q}^2.
\end{aligned}
\]
Here $\mfk S_4$ denotes the set of permutations of $\{1,2,3,4\}$.
We observe that both $b_{\sform,q}$ and $c_{\sform,q}$ are symmetric polynomial combinations of the components of $\chipp(\ua)$. They are both homogeneous of degree $8$, so they are invariant with respect to multiplying the entries of $\chipp(\ua)$ by some $\lambda\in\mu_4$. Moreover, we notice that both $b_{\sform,q}$ and $c_{\sform,q}$ are invariant under conjugation. Therefore $b_{\sform,q}$ and $c_{\sform,q}$ depend only on the class $[\chipp(\ua)]\in\mu_4^4/\!\sim$. Now, a straightforward computation gives the values listed in \cref{table} in all the 8 cases. 
The proposition follows, since moreover $d_{\sform,q}=2\Re(\wbchipp(a_1a_2a_3a_4)\pi^2_{4,q})$.
\end{proof}

We remark that in the above statements we have $\rhop,\rhopp\in [-1,1]$ for all $p\equiv 1 \pmod 3$ and $q\equiv 1 \pmod 4$, because $\abs{\pi_{3,p}}=\sqrt p$ and $\abs{\pi_{4,q}^2}=q$. In fact, in the next sections we are going to use the fact that for all $\rho\in (-1,1)$ the inequalities $\rhop\less \rho$ and $\rhopp\less\rho$ are satisfied for a positive proportion of the primes. Notice moreover that $\chipp(-1)=1$ if $q\equiv 1\pmod 8$ and  $\chipp(-1)=-1$ if $q\equiv 5\pmod 8$. Therefore a necessary condition to have $\sol_\sform(0,q)\less q^3$ in the case $\s=k=4$, is that $\Kpp=b_{\sform,q}\pm c_{\sform,q}\less 2$ for some choice of sign $\pm$. Compare this with \cref{table}.

\section{Hecke L-functions and asymptotic estimates}\label{sec:hecke}

There is a universal strategy, which we will implement later, to study the range of values of $H_{F,p}$ and $H_{F,q}$ from the previous section, or more generally quantities likewise computed from Jacobi sums. 
In this section we collect the main ingredients of the method: following Weil the Jacobi sums can be interpreted as Hecke characters; the theory of Hecke L-functions provides ``generalized prime number theorem''-type estimates; finally these estimates are feeded into equidistribution lemmas. 
This game plan is inspired by Moreno \cite{Moreno}, even though in detail we follow more closely an approach of Heath-Brown and Patterson \cite[p.115]{HeathPatterson} by using the generalized prime number theorem of Kubilyus and the equidistribution lemma of Erd{\H{o}}s and Tur{\'a}n. 

\subsection{Hecke characters}\label{sec:hecke:L}

Let $K$ be a number field of degree $d:=[K:\Q]$. 
A \emph{Hecke character} (also named Gr{\"o}ssencharakter) of $K$ is a character of the id\`ele class group $\mbb A_K^\times / K^\times$. 
More down to earth, let $\O_K$ be the ring of integers of $K$, let $\m\subseteq\O_K$ be a nonzero ideal and let $\I_\m$ be the set of the ideals of $\O_K$ that are coprime to $\m$. 
Since $\O_K$ is a Dedekind domain, $\I_\m$ is a multiplicative monoid generated by the prime ideals of $\O_K$ that don't divide $\m$. 
A multiplicative homomorphism
\[
 H:\ \I_\m\to\C^\times
\]
is a Hecke character of $K$ if there is a continuous group homomorphism $\chi_\infty: (K\otimes_\Q \R)^\times \to \C^\times$ such that $H((\alpha))=\chi_\infty(\alpha\otimes 1)$ for all $\alpha\in \O_K$ satisfying $\alpha\equiv 1 \pmod \m$. 
In other words, $H$ is a Hecke character if, for the same $\alpha$,
\begin{equation}\label{eq:exponents}
H((\alpha))=\prod_{\si: K\to \C} \si(\alpha)^{k_{\si}}\abs{\si(\alpha)}^{c_\si}
\end{equation}
for some integers $(k_\si)_\si$ and complex numbers $(c_\si)_\si$. 
We say that the $2n$-tuple $(k_\si,c_\si)_\si$ is a \emph{vector of exponents} of $H$. 
The ideal $\m$ is a \emph{defining ideal} of $H$ and $\chi_\infty$ is the \emph{infinity type} of $H$.
A Hecke character $H$ is unitary if $\abs{H(\mfk a)}=1$ for all $\mfk a\in\I_\m$. 

As a word of caution, we mention the fact that some authors define $\chi_\infty^{-1}$ to be the infinity type of $H$. 
Moreover, sometimes in the literature the Hecke characters are required to be unitary by definition, while those that are not unitary are called quasicharacters. 
For more details on the basic facts and properties of Hecke characters, we refer to the fundational article of Hecke \cite{Hecke} or to the first chapter of Kubilyus \cite{Kubilyus}. 

According to the general theory, we know that the unitary  Hecke characters of $K$ with defining ideal $\m$ form a finitely generated abelian group $G(K,\m)$. 
This group contains a natural free subgroup $G^{(1)}(K,\m)$ of order $d-1$ whose elements are called \emph{Hecke characters of the first kind}. 
Then the group $G(K,\m)$ of all unitary Hecke characters (which in the literature are sometimes called Hecke characters \emph{of the second kind} \cite{Kubilyus}) splits into a direct product 
\[
G(K,\m) = G^{(1)}(K,\m)\times T(K,\m),
\]
where $T(K,\m)$ is the subgroup of Hecke characters of finite order, which are sometimes called \emph{abelian characters}. 
This is a finite subgroup with cardinality $h(\m):= 2^{r_1} h(K) \varphi(\m)$, where $r_1$ is the number of real embeddings of $K$, $h(K)$ is the ideal class number and $\varphi(\m):=\#(\O_K/\m)^\times$ is the Euler function. 
In other words, every unitary Hecke character $H$ can be written uniquely as a product $H=\chi\xi$ of an abelian character and a Hecke character of the first kind.  
Every Hecke character can be normalized to a unitary one through multiplication by a real power of the norm character. 

In this paper we are concerned with two important examples of Hecke characters.

\begin{definition}\label{def:symbol}
Let $\s\in\N_+$, let $K$ be a number field containing all $\s$-th roots of unity, with ring of integers $\O_K$, and let $a\in\O_K\setminus\{0\}$. 
Let also $\m_1:=(a\s)\subseteq \O_K$ and $\m_2:=(\s)\subseteq \O_K$. 
Recall the definition of $\chi_{\s,\p}(\cdot)$ from \cref{sec:char:power} and that $\chi_{\s,\p}$ can be seen as a character of the finite field $\O_K/\p$. 
Then we define the power residue symbol and the (normalized) Jacobi sum symbol
\[
\left(\frac a \cdot\right)_{\s}:\ \I_{\m_1}\to \C^\times
\ ,\quad\ 
\J_{\s}(\cdot):\ \I_{\m_2}\to \C^\times
\]
by setting $\left(\frac a \p\right)_{\s}:=\chi_{\s,\p}(a)$ and $\J_{\s}(\p):=-J(\chi_{\s,\p},\chi_{\s,\p})(N\p)^{-1/2}$ for all prime ideal $\p$ coprime to $\m_1$ for the first, to $\m_2$ for the second, and then extending by multiplicativity. 
Here $N\p:=\#(\O_K/\p)$ denotes the norm of $\p$.
\end{definition}

\begin{proposition}\label{hecke}
Keep the notation of \cref{def:symbol}. 
\begin{enumerate}[(i)]
\item The power residue symbol $\left(\frac a \cdot\right)_\s$ is a unitary abelian character of $K$ with trivial infinity type and with $\m_{a,\s}:=(a\s)^{f_{a,\s}}$ as a defining ideal, for some  $\f_{a,\s}\in\N_+$. 
\item The Jacobi sum symbol $\J_{\s}(\cdot)$ is a Hecke character of $K$ with defining ideal $\m_{\J_\s}:=(\s^2)$. It is unitary if $s\geq 3$. 
Moreover if $s\in\{3,4\}$ then the infinity type of $\J_{\s}(\cdot)$ satisfies $\chi_\infty(\alpha\otimes 1)= \alpha / {\abs{\alpha}}$ for all $\alpha\in K^\times$.
\end{enumerate}
\end{proposition}

\begin{proof}
	Statement (i) is a consequence of Class Field Theory \cite[Theorem 1.13(8) in Ch. 2.\S~1.8, and Example 36 in Ch. 1.\S~6.3]{Koch}. 
	The assertions in (ii) follow instead from the work of Weil \cite{Weil}, as follows. 
	The fact that $\J_{\s}(\cdot)$ is a Hecke character is the main theorem of that paper: notice in particular that the minus sign in the definition of $\J_{\s}(\cdot)$ reflects the different sign convention for Jacobi sums in Weil's paper \cite[eq. (I)]{Weil} and in ours (\cref{def:Jacobi}). 
	The fact that $\J_{\s}(\cdot)$ is unitary for $s\geq 3$ follows from \cite[eq. (10)]{Weil}. 
	To compute the infinity type, Weil gives explicit general formulas in \cite[eq.(9) and the bottom of p.491]{Weil}. 
	According to these formulas, if $s\in\{3,4\}$, we get that 
	\[
	\J_{\s}((\alpha))= \bar{\alpha}^{\omega_1}\alpha^{\omega_2}N((\alpha))^{-1/2}
	\]
	for every $\alpha\in \O_K$ such that  $\alpha\equiv 1 \bmod {s^2}$, where 
	\[
	\omega_i=\left\lfloor \frac i s + \frac i s \right\rfloor.
	\]	
	Then, we have $\omega_1=0$ and $\omega_2=1$. 
	Since $N((\alpha))=\alpha \bar\alpha = \abs{\alpha}^2$, the claim is proved.    
\end{proof}

\subsection{Hecke L-functions}

Given 
a Hecke character $H$ of $K$, one considers the attached Hecke L-function 
\[
L(H,s) = \sum_{\mfk a\in \I_\m} H(\mfk a) (N\mfk a) ^{-s} = \prod_{\p\in\I_\m\cap\op{Spec}\O_K} (1 - H(\p)(N\p)^{-s})^{-1},
\]
where $\op{Spec} \O_K$ is the set of prime ideals of $\O_K$. 
Hecke L-functions form a class of relatively well-behaved L-functions. 
If $H$ is unitary then both the Dirichlet series and the infinite Euler product above converge absolutely on the right half plane $\Re(s)\geq 1$.
Moreover $L(H,s)$ has a meromorphic analytic continuation on all the complex plane, which is entire if $H$ is nontrivial. 

Analytic estimates for $L(H,s)$ can be given in terms of the ``size'' of the character $H$. 
Following Kubilyus \cite{Kubilyus} we fix arbitrarily a basis $\xxi = (\xi_1,\ld,\xi_{d-1})$ of the group of Hecke characters of the first kind, so that every Hecke character $H$ can be written uniquely as 
\begin{equation}\label{Hecke:basis}
	H = \chi \xi_1^{m_1}, \ld, \xi_{d-1}^{m_{d-1}},
\end{equation}
for some abelian character $\chi$ and some integers $m_i\in\Z$. 
Then we define the \emph{size of $H$ with respect to $\xxi$} by 
\begin{equation} \label{def:size}
v_\xxi(H):= \prod_{i=1}^{d-1}(\abs {m_i}+3). 
\end{equation}

With this notation, a classical result concerning zero-free regions of Hecke L-functions states that $L(H,s)\neq 0$ if $\s=\sigma + it$ satisfies
\begin{equation}\label{eq:zerofree}
\sigma \great 1-\frac{c(K,\xxi)}{\log (\abs{t}+3) + \log v_\xxi(H)},
\end{equation}
for some constant $c(K,\xxi)$ independent of $H$ \cite[Lemma 2]{Kubilyus}. 
%
In fact, more recent results for zero-free regions of Hecke L-functions are available, which provide more precise estimates than \eqref{eq:zerofree} both in the $v_\xxi$ and $t$ aspects \cite{Coleman,zeroregion} (see also \cref{rmk:Hecke:size}). 
For more about the theory of (Hecke) L-functions, see \cite{Hecke}, \cite[Chapter 5.10]{IK}, \cite{Lang} or \cite{MurtyBook}. 

\begin{remark}\label{rmk:Hecke:size}
	In the literature there is no universally accepted notation for the ``size'' of an Hecke character. 
	For example, Coleman \cite{Coleman} defines it as the $L^2$-norm of some suitable vector of exponents $(k_\si,c_\si)_\si$ of the Hecke character, while Mitsui \cite{Mitsui} uses a quantity related to the $L^1$-norm of this vector of exponents. 
	In the book of Iwaniec and Kowalski \cite{IK}, instead, the role of $v_\xxi(H)$ is played by the analytic conductor $\qq(H)$. 
	The analytic conductor is a quantity that is computed in terms of the norm of the algebraic conductor of $H$ (which is the largest defining ideal of $H$, with respect to inclusion) and the $\gamma$-factors of the functional equation of the $L(H,s)$. 
	In fact, all these notions are related. 
	For example, Hecke \cite{Hecke} describes the Hecke characters of the first kind by means of explicit formulas, which themselves are given in terms of the choice of a basis for the units of $\O_K$. 
	From such description one can explicit a choice of a basis $\xxi$ for the set of Hecke characters of the first kind. 
	Furthermore, Hecke provides explicit formulas for the $\gamma$-factors of $L(H,s)$. 
	From these formulas it is possible to verify that 
	\[
	\log v_\xxi(H) \asymp \log \qq(H),
	\]
	where the implied constant may depend on $K$ and $\m$, but is independent of $H$. 
	Similar considerations apply to the ``sizes'' defined by Coleman and Mitsui. 
\end{remark}

\subsection{Asymptotic estimates}\label{sec:asymp}

We are interested in Hecke characters primarily because they give access to the following version of the prime number theorem \cite[Lemma 4]{Kubilyus}.  
\begin{lemma}\label{pnt}
Let $\m$ be an ideal of $\O_K$ and let  $\xxi$ be a basis of $G^{(1)} (K,\m)$. 
Then there are effective constants $c_1(K,\xxi), c_2(K,\xxi)\great 0$ 
such that for each nontrivial unitary Hecke character $H\in G(K,\m)$ and every $T\geq 2$ we have 
\begin{equation}\label{eq:pnt}
\left|   
\sum_{\substack{\p\in\I_\m\cap \op{Spec}\O_K\\N\p\leq T}}  H(\p) 
\right|
    \leq 
c_1(K,\xxi) \,T\, \exp\left( \frac{-2c_2(K,\xxi)\, {\log T}}{\log v_\xxi(H) + \sqrt {\log T}}\right).
\end{equation}
\end{lemma}
\Cref{pnt} is proved via standard arguments concerning zero-free regions of Hecke L-functions \cite{Hecke} \cite[Thm 5.13]{IK}, using \eqref{eq:zerofree}. 
Refinements can be given using the more precise estimates for zero-free regions due to Coleman et al.\  \cite{Coleman,zeroregion}.  

Assuming that $v_\xxi(H)\leq \sqrt {\log T}$, 
the expression on the right-hand side of \eqref{eq:pnt} simplifies to 
\[
c_1 T e^{- c_2 \sqrt{\log T}}. 
\]
The strength of \cref{pnt} is appreciated by noticing that the number of summands in the left-hand side of \eqref{eq:pnt} is asymptotic to $T/\log T$ by a classical theorem of Landau.  

We now collect some estimates that can be easily checked by partial integration-summation (e.g. \cite[Thm 421, 22.5.2]{HardyWright}). It is useful, in order to simplify the calculations and the final estimates, to use the fact that 
\begin{equation}\label{eq:error:log}
e^{- \alpha \sqrt{\log T}} = o((\log T)^{-A}),
\end{equation}
for any fixed $A,\alpha\great 0$ and for $T\to\infty$.

\begin{lemma}\label{asymptotic}
Let $\mcl A\subseteq \N$ be a set of positive integers such that for $T\to \infty$ the following estimate holds, for some $c,d\great 0$,  and where $\Li(T):=\int_{2}^{T} dx/\log x$:
\[
\#\mcl A \cap [1,T] = c \Li(T) + O(Te^{- d \sqrt{\log T}}).
\]
Then:
\begin{alignat}{2}
\label{eq:asymp:log}
	& \sum_{p\in\mcl A\cap [1,T]} \log p 
	&&= (c+o(1)) T;%
\\
\label{eq:asymp:1/2} 
	&\sum_{p\in\mcl A\cap [1,T]} p^{-1/2} 
	&&= (2c+o(1)) \frac{\sqrt T}{\log T};
\\
\label{eq:asymp:1}
	&\sum_{p\in\mcl A\cap [1,T]} p^{-1} 
	&&= c\log\log T + O(1);%
\\
\label{eq:asymp:3/2}
	&\sum_{p\in\mcl A\cap [1,T]} p^{-3/2} 
	&&= O(1);%
\\
\label{eq:asymp:2}
	&\sum_{p\in\mcl A\cap [1,T]} p^{-2} 
	&&= O(1).
\end{alignat}
\end{lemma}


\subsection{Equidistribution}\label{sec:ET}

The estimates coming from \cref{pnt} will be used to show that the values of some Hecke characters equidistribute on the unit circle. Classical tools to prove such results are Weyl's equidistribution lemma or its quantitative version due to Erd{\H{o}}s and Tur{\'a}n \cite[Theorem III]{ET}. The following proposition is a direct consequence of the Erd{\H{o}}s-Tur{\'a}n equidistribution lemma.

\begin{lemma}\label{ET}
Let $\{h_a\}_{a\in\mcl A}$ be a sequence of complex numbers of modulus 1 indexed by a finite set $\mcl A$ and for every  $n\in\N_+$ let $S_n:=\sum_{a\in\mcl A} \Re(h_a^n)$. Let $\phi_1,\phi_2$ be real numbers satisfying $0\leq \phi_1\less\phi_2\leq \pi$. 
Then 
\[
\#\{a\in\mcl A:\ \Re h_a \in [\cos\phi_2,\cos\phi_1]\} = \frac {\phi_2-\phi_1}{\pi}\#\mcl A + E
\] 
with 
\[
\abs{E} \leq C\left( \frac{\#\mcl A}{N} + \sum_{n=1}^N \frac 1 n \abs{S_n}\right)
\]
for every $N\in\N_+$ and for an absolute constant $C\great 0$.
\end{lemma}

See \cite[Chapter 5]{IK} for a general discussion on L-functions and equidistribution and \cite[Exercise 3.2]{MurtyBook} for more precise versions of the  Erd{\H{o}}s-Tur{\'a}n inequality. 
In this article, the above results will be used to show the equidistribution of $\Hp$ and $\Hpp$ of \cref{erre:modp:3,erre:modp:4} as $p,q$ vary. 
In other words, equidistribution of Jacobi sum symbols at prime elements. 
We remark that there are also equidistribution results for Gauss sums, which in turn are related to a famous problem of Kummer \cite{Moreno,HeathPatterson,Patterson}. 

%
%

\section{Exceptional forms and the term $\Kpp$}\label{sec:few}
\subsection{Exceptional biquadratic diagonal forms}\label{sec:exceptional}

In this paragraph we study a special family of biquadratic diagonal forms.

\begin{definition}\label{def:exceptional}
We say that a biquadratic diagonal form $\form$ is exceptional if there are positive integers $a,b,c_1,c_2,c_3,c_4$ and a permutation $\sigma\in\mfk S_4$ such that 
\[\form= a(c_1x_{\si(1)})^4+b(c_2x_{\si(2)})^4+4a(c_3x_{\si(3)})^4+4b(c_4x_{\si(4)})^4.\]
\end{definition}

We will prove the following characterization of exceptional forms.
\begin{theorem}\label{thm:exceptional:iff}
A biquadratic diagonal form $\form$ is exceptional if and only if for all prime numbers $q\not\in \Sigma_\sform$ we have $\sol_\sform(0,q)\geq q^3$.
\end{theorem}

We first show that the condition is necessary through the following two lemmas which treat separately the cases $q\equiv 1 \pmod 4$ and $q\equiv 3\pmod 4$. 
The proof of sufficiency is posponed to \cref{sec:main}.


\begin{lemma}\label{1144:1mod4}
Let $\form= a(c_1x_{\si(1)})^4+b(c_2x_{\si(2)})^4+4a(c_3x_{\si(3)})^4+4b(c_4x_{\si(4)})^4$ be an exceptional form and let $q$ be a prime number with $q\equiv 1 \pmod 4$ and $q\ndivides ab c_1c_2c_3c_4$. 
Then $\sol_\sform(0,q) \geq q^3$.
\end{lemma}
\begin{proof}
Since $\#\FF_q^\times$ is divisible by four, $\FF_q$ contains a fourth root of unity $\omega\in\FF_q$ with $\omega^2=-1$. Let $\lambda:=1+\omega$, and notice that $\lambda^4=-4$.
Consider now $\formm= ax_1^4+bx_2^4-ax_3^4-bx_4^4$.
Then the map $(x_1,\ld,x_4)\mapsto (c_1x_{\si(1)},c_2x_{\si(2)},\lambda c_3x_{\si(3)}, \lambda c_4x_{\si(4)})$ gives a bijection between $\Sol_\sform(0,q)$ and $\Sol_{\sformm}(0,q)$.

For every $t\in\FF_q$ define $n_t:=\#\{(y,z)\in\FF_q^2: ay^4+bz^4 = t\}$. 
Then we deduce that 
\[
\sol_\sform(0,q) = \sol_{\sformm}(0,q) = \sum_{t\in\FF_q} n_t^2.
\]
However, it is clear that  $\sum_{t\in\FF_q} n_t = q^2$. 
Hence, from the quadratic-arithmetic mean inequality (or Cauchy-Schwartz) we get $\sol_\sform(0,q)\geq q^3$.
\end{proof}

\begin{lemma}\label{1144:3mod4}
Let $\form= a_1 x_1^4+\ld + a_4 x_4^4$ with $a_1,\ld,a_4\in\Z\setminus \{0\}$ and let $q$ be a prime number with $q\equiv 3 \pmod 4$ and $q\ndivides a_1a_2a_3a_4$. 
Then $\sol_\sform(0,q) = q^3 + \left(\frac{a_1a_2a_3a_4}{q}\right) q(q-1)$, where $\left(\frac{\cdot}{q}\right)$ denotes the Legendre symbol.
\end{lemma}
\begin{proof}
Recall that the Legendre symbol $\chippl(\cdot\bmod q):=\left(\frac{\cdot}{q}\right)$ is the only nontrivial quadratic character of $\FF_q$, so: $\X_q^{(2)} = \{\chippl\}$.
For $a\in\FF_q^\times$ we have $\chippl(a)=1$ if and only if $a$ is a quadratic residue modulo $q$, and we have $\chippl(a)=-1$ otherwise.

Let $\formm: a_1 x_1^2+\ld + a_4 x_4^2$ be a quadratic form with the same coefficients as $\form$. 
Since $J(\chippl)=1$, we get $G(\chippl)^2 = \chippl(-1) q$ by (3.1).
Then by (3.2) we get $J_0(\chippl,\chippl,\chippl,\chippl) = (q-1)q$, and since $\X_q^{(2)} = \{\chippl\}$ we see that
\[
\sol_{\sformm}(0,q) = q^3 + \chippl(a_1a_2a_3a_4) (q-1) q,
\]
by \cref{erre:Jacobi} and multiplicativity of $\chippl$.
Finally, we notice that $\sol_{\sform}(0,q) = \sol_{\sformm}(0,q)$, because, since $q\equiv 3\pmod 4$, we have $\#\{x\in\FF_q:\ x^4=y\} = \#\{x\in\FF_q:\ x^2=y\}$ for all $y\in\FF_q$. 
\end{proof}

If $\form= a(c_1x_{\si(1)})^4+b(c_2x_{\si(2)})^4+4a(c_3x_{\si(3)})^4+4b(c_4x_{\si(4)})^4$  is an exceptional form, the product of its coefficients is a perfect square. 
Then \cref{1144:3mod4} implies that $\sol_\sform(0,q) \geq q^3$ if $q\equiv 3\pmod 4$ and $q\not \in\Sigma_\sform$. 
Together with \cref{1144:1mod4} we conclude that $\sol_\sform(0,q) \geq q^3$ for every $q\not\in\Sigma_\sform$, as claimed in \cref{thm:exceptional:iff}.

%
%

\subsection{Computing $\Kpp$ via Kummer's theory}\label{sec:kummer}

In order to prove that the condition in \cref{thm:exceptional:iff} is sufficient, we need to analyze in more detail the formula given in \cref{erre:modp:4}.
Fix $\form= a_1x_1^4 + a_2x_2^4  + a_3x_3^4 + a_4x_4^4$ with $a_1,a_2,a_3,a_4\in\Z\setminus\{0\}$ and recall that $\mu_4=\{1,-1,i,-i\}$. 
We notice that the term $\Kpp$ in \eqref{eq:modp:4} depends only on $\chipp(-1)$ and $\chipp(a_1),\ld,\chipp(a_4)$, 
and that the character $\chipp$ depends on the choice of a prime ideal $\pp$ of $\Z[i]$ above $q$. 
A prime $q\equiv 1 \bmod 4$ splits in $\Z[i]$ as $q=\pp \wb\pp$ and we have $\chi_{4,\wb \pp} = \wb{\chi_{4,\pp}}$. 

Let $\Pquattro$ denote the set of prime numbers $q$ that satisfy $q\equiv 1\bmod 4$ and $q\not\in\Sigma_\sform$. 
If $q\in\Pquattro$, let $\chipp(\ua,-1)\in\mu_4^4\times\{\pm 1\}$ be a shorthand for $((\chipp(a_1),\ld,\chipp(a_4)),\chipp(-1))$. 
For all $\u\in\mu_4^4\times\{\pm 1\}$ let  $\wb\u\in\mu_4^4\times\{\pm 1\}$ be obtained from $\u$ by componentwise complex conjugation and let 
\[
\P_{\sform,\u} := \{q\in\P_{\sform,1}:\ \chipp(\ua,-1) \in\{\u, \wb\u\}\}.
\]
The natural setting to study these sets is over the Gaussian quadratic field, via Kummer's theory. 
Let $K=\Q(i)$ and let $\Delta_\sform\subseteq K^\times/(K^\times)^4$ be the (finite abelian) subgroup multiplicatively generated by $a_1, a_2,a_3,a_4,-1$. 
Notice that $-1 \bmod (K^\times)^4 = 4 \bmod (K^\times)^4$, because $(1+i)^4=-4$.
Moreover, observe that $(K^\times)^4 \cap \Q_+ = (\Q^\times)^4$, where $\Q_+$ denotes the multiplictive group of strictly positive rational numbers.
Therefore we can view $\Delta_\sform$ as the subgroup of $\Q_+/(\Q^\times)^4\subseteq K^\times/(K^\times)^4$  multiplicatively generated by $a_1, a_2,a_3,a_4,4$. 
Notice that $\Q_+/(\Q^\times)^4\cong \bigoplus_{\ell \text{ prime}} \Z/4\Z$ as an abelian group.

Let $L=K(\sqrt[4]{\Delta_\sform})$. 
By Kummer's theory \cite[Ch. I.\S~5]{NeukirchCFT} 
 we have that $L/K$ is a finite abelian extension of exponent 4 with Galois group $G:=\op{Gal}(L/K)\cong \op{Hom}(\Delta_\sform,\mu_4)$. 
The isomorphism $\psi: G\to \op{Hom}(\Delta_\sform,\mu_4)$ and the dual $\hat\psi: \Delta_\sform\to \op{Hom}(G,\mu_4)$ are induced by the perfect pairing $G\times\Delta_\sform\to\mu_4$ given by $(\sigma,a)\mapsto \frac {\sigma(\sqrt[4]{a})}{\sqrt[4]{a}})$. 
The link with the power residue characters is given by the fact that 
\[
\left(\frac a \p\right)_4 = \frac{(\p,L/K)(\sqrt[4]{a})}{\sqrt[4]{a}}
\]
for all $a\in \O_K \cap (L^\times)^4$ and all prime ideal $\p\subseteq \O_K$ coprime with $ma$ where  $m=2a_1\dots a_4$.
Here $(\p,L/K)\in \op{Gal}(L/K)$ denotes the Frobenius element of $\p$, which is well-defined because $L/K$ is abelian.
In other words, the values of $\chi_{4,\p}$ on $a_1,a_2,a_3,a_4,-1$ are obtained by applying $\hat\psi(a_1),\ld,\hat\psi(-1) \in  \op{Hom}(G,\mu_4)$ to the Frobenius element $(\p,L/K)\in G$. 
Or dually, by applying $\psi((\p,L/K))\in \op{Hom}(\Delta_\sform,\mu_4)$ to $a_1,a_2,a_3,a_4,-1\in \Delta_\sform$.


\subsection{The sets $\P_{\sform,\u}$ and Chebotarev's theorem}\label{sec:chebotarev}

Following the discussion in \cref{sec:kummer}, we consider the map 
\begin{center}
\begin{tabular}{>{$}r<{$}>{$}c<{$}>{$}r<{$}>{$}c<{$}}
\aphi_\sform: \ 	&  \op{Hom}(\Delta_\sform,\mu_4) 	& \too 		& \mu_4^4\times\{\pm 1\}\\
			& 				\chi		& \longmapsto &((\chi(a_1),\ld,\chi(a_4)),\chi(-1))
\end{tabular}
\end{center}

\begin{proposition}\label{chebotarev}
Let $\form=a_1 x_1^4 + a_2x_2^4+a_3 x_3^4 + a_4 x_4^4$ with $a_1,\ld,a_4\in\Z\setminus\{0\}$ and let $\u\in\mu_4^4\times\{\pm 1\}$ be in the image of $\aphi_\sform$. 
Then $\P_{\sform,\u}\neq \emptyset$ and moreover for $T\to \infty$ we have
\begin{equation}\label{eq:chebotarev}
\#\P_{\sform,\u}\cap [1,T] = \d \Li(T)+ O(Te^{-\a\sqrt{\log T}}) 
\end{equation}
for some $\d\geq \frac 1 {1024}$ and some effectively computable absolute constant $\a\great 0$.
\end{proposition}

\begin{proof}
Denote for brevity $\aphi=\aphi_\sform$ and recall that we described an isomorphism $\psi:\ \op{Gal}(L/K) \to \op{Hom}(\Delta_\sform,\mu_4)$ in \cref{sec:kummer}. 
Then by Chebotarev's theorem \cite[Thm. 3.4]{Serre} the set
\[
\P:=\{\p\in\op{Spec}{\O_K}:\ \psi((\p,L/K)) \in \aphi^{-1}(\{\u,\wb\u\})\}
\]
satisfies 
\begin{equation}\label{eq:c:e:1}
\#\{\p\in \P:\ N\p\leq T\} = \d' \Li(T)+ O(Te^{-\a\sqrt {\log T}}),
\end{equation}
for some $\a\great 0$ and $\d'=\frac {\# \aphi^{-1}(\{\u,\wb\u\})}{\#\Delta_\sform}$. 
Since $\Delta_\sform$ is an abelian group generated by 4 elements of order at most 4, and an element of order 2, we have $\#\Delta_\sform\leq 512$. 
In particular the degree of $L=\Q(i,\sqrt[4]{\Delta_\sform})$ over $\Q$ is at most 1024, and so we can take $\a$ to be an effectively computable absolute constant by \cite{Chebotareveffective}. 
For the sake of completeness, we remark that also the constant implied in the $O$-notation can be computed effectively, and it is an absolute constant if the Dirichlet zeta function of $L$ has no real zero, while it may depend on the discriminant of $L$ otherwise. 
Since $\u$ is in the image of $\aphi$, we have $\# \aphi^{-1}(\{\u,\wb\u\})\geq 1$ and so $\d'\geq \frac 1 {512}$.  
Notice that for every $T$ there are at most $\sqrt T$ primes $\p\in\O_K$ of degree two with $N\p\leq T$. 
Indeed, these are the primes of the form $\p=p\O_K$ where $p$ is a (rational) prime number with $p\equiv 3\pmod 4$, and $N\p=p^2$. 
Therefore the estimate in \eqref{eq:c:e:1} is also valid when we restrict to the primes of degree 1 which are coprime with $2a_1a_2a_3a_4$. These come in conjugate pairs, which correspond bijectively to rational primes $q\in\Pquattro$ via $q=\pi_{4,q}\wb{\pi}_{4,q}$. For such $q$ we have
\[
\pi_{4,q}\in\P\iff \wb\pi_{4,q}\in\P\iff q\in\P_{\sform,\u},
\]
therefore we get \eqref{eq:chebotarev} with $\d=\d'/2\geq \frac 1 {1024}$. 
\end{proof}

\subsection{Characters of $\Delta_\sform$, exceptional forms and the inequality $\Kpp\leq 1$}\label{sec:galois}

In this paragraph we finally compute the term $\Kpp$ of \cref{erre:modp:4} when $\form$ is not exceptional and we deduce, together with \cref{chebotarev}, that $\Kpp\leq 1$  for a positive proportion of the primes $q\ndivides \Sigma_\sform$. 
Let $\bphi_\sform: \ \op{Hom}(\Delta_\sform,\mu_4) \longrightarrow (\mu_4^4/\!\sim)\times\{\pm 1\}$ be the composition of $\aphi_\sform$ with the natural projection $\proj: \mu_4^4\times\{\pm 1\} \to (\mu_4^4/\!\sim)\times\{\pm 1\}$. See \cref{sec:diagonal:sol} for the definition of $\mu_4^4/\!\sim$.
For brevity, we denote the elements of $\mu_4^4/\!\sim$ by $U_1,\ld,U_8$ as shown in \cref{table}.

\begin{lemma}\label{kummer:exceptional}
Let $\form=a_1 x_1^4 + a_2x_2^4+a_3 x_3^4 + a_4 x_4^4$ with $a_1,\ld,a_4\in\Z\setminus\{0\}$. 
Assume that, in the image of $\bphi_\sform$, there is no element $(U,u_5)$ with
\begin{equation}\label{eq:c:e:0}
(U,u_5) \in \{(U_i,1):\ i\in\{2,3,5,7,8\}\} \cup \{(U_i,-1):\ i\in\{1,2,5,6,7\}\}.
\end{equation}
Then $\form$ is exceptional.
\end{lemma}

\begin{proof}
We notice that the image of $\bphi_\sform$ doesn't change, if we multiply one coefficient of $\form$ by the fourth power of an integer, or if we multiply all its coefficients by the same nonzero integer, or if we permute its cofficients. 
Therefore we may assume without loss of generality that $\op{gcd}(a_1,a_2,a_3,a_4)=1$ and that none of $a_1,\ld,a_4$ is divisible by nontrivial fourth powers. 

For every prime $\ell$ we consider the group homomorphism $\chi'_\ell:\ \Q_+/(\Q^\times)^4\to \mu_4$ given by $ r \mapsto i^{v_\ell(r)}$, where $v_\ell(\cdot)$ is the $\ell$-adic valuation. 
Let $\chi_\ell=(\chi'_\ell)|_{\Delta_\sform}\in\op{Hom}(\Delta_\sform,\mu_4)$ be the restriction of $\chi'_\ell$ with respect to the inclusion $\Delta_\sform\hookrightarrow \Q_+/(\Q^\times)^4$.

We cannot have $\bphi_\sform(\chi_2) = (U_3,-1)$, otherwise $\bphi_\sform(\chi_2^2) = (U_2,1)$ is in the image of $\bphi_\sform$. Since $\chi_2(-1)=\chi'_2(4)=-1$, we must have $\bphi_\sform(\chi_2)\in\{(U_4,-1), (U_8,-1)\}$. 
By the remarks made at the beginning of the proof, we may therefore assume that either
\begin{enumerate}[(a)]
\item $\form= d_1 x_1^4 + d_2 x_2^4 + 4 d_3 x_3^4 + 4 d_4 x_4^4$, or
\item $\form= d_1 x_1^4 + 2d_2 x_2^4 + 4 d_3 x_3^4 + 8 d_4 x_4^4$,
\end{enumerate}
for some odd integers $d_1,d_2,d_3,d_4$ with $\op{gcd}(d_1,d_2,d_3,d_4)=1$, none of which is divisible by nontrivial fourth powers. 

Notice that for a prime number $\ell\neq 2$ we must have $\bphi_\sform(\chi_\ell)\in\{(U_1,1),(U_4,1), (U_6,1)\}$. 
This means that $\ell$ doesn't divide $d_1d_2d_3d_4$ or else there are exactly two indices $i,j\in\{1,\ld,4\}$ such that $\ell\divides d_i$ and $\ell\divides d_j$, and moreover $v_\ell(d_i)=v_\ell(d_j)\in\{1,2,3\}$.

Suppose that $\form$ is not exceptional. 
Then observe that one of the following cases must occur, for some distinct odd prime numbers $p_1,p_2$:
\begin{enumerate}[(i)]
\item $p_1$ divides both $d_1$ and $d_2$;
\item $p_1$ divides both $d_3$ and $d_4$;
\item $p_1$ divides $d_j$ and $d_3$, $p_2$ divides $d_j$ and $d_4$ for some $j\in\{1,2\}$;
\item $p_1$ divides $d_1$ and $d_j$, $p_2$ divides $d_2$ and $d_j$ for some $j\in\{3,4\}$.
\end{enumerate}
We define auxiliary values $\aux(1)=\aux(3)=2$ and $\aux(2)=1$. 
Now for each case (i)-(iv) we consider the following auxiliary character $\achi\in\op{Hom}(\Delta_\sform,\mu_4)$: 
in (i) $\achi=\chi_{p_1}^{\aux(v_{p_1}(d_1))}$; 
in (ii) $\achi=\chi_{p_1}^{\aux(v_{p_1}(d_3))}$; 
in (iii) and (iv) $\achi=\chi_{p_1}^{\aux(v_{p_1}(d_j))}\chi_{p_2}^{\aux(v_{p_2}(d_j))}$. 
In each case we get that $\aphi_\sform(\achi)$ is equal to $((-1,-1,1,1),1)$ or $((1,1,-1,-1),1)$. 
But then we see that $\bphi_\sform(\chi_2\achi)=(U_1,-1)$ in case (a) above, and  $\bphi_\sform(\chi_2\achi)=(U_6,-1)$ in case (b). 
Both are contrary to our assumptions, so $\form$ is exceptional.
\end{proof}

\begin{proposition}\label{K:exceptional}
Let $\form$ be a biquadratic diagonal form that is not exceptional. 
Choose $\u\in\mu_4^4\times\{\pm 1\}$ in the image of $\aphi_\sform$ such that $\proj(\u)\in (\mu_4^4/\!\sim)\times\{\pm 1\}$ satisfies \eqref{eq:c:e:0}.  
Then $\Kpp\leq 1$ for all $q\in \P_{\sform,\u}$.
\end{proposition}
\begin{proof}
Notice that $\proj(\u)=\proj(\wb\u)$, so for all $q\in\P_{\sform,\u}$ we verify from \eqref{eq:c:e:0} and \cref{table} that $\Kpp\leq 1$.
\end{proof}

%
%

\section{Equidistribution and the terms $\Hp$ and $\Hpp$}\label{sec:equidistribution}

In this section we investigate the remaining terms $\rhop,\rhopp$ in \cref{erre:modp:3} and \cref{erre:modp:4}. 
The main fact that we exploit is that $\Hp$ and $\Hpp$ essentially take the values of infinite order unitary Hecke characters. 
This enables us to show that they equidistribute on the unit circle as $p,q\to \infty$, using \cref{ET} and the estimates given by \cref{pnt}.

\subsection{Equidistribution of $\Hp$}\label{sec:equidistribution:3}

The case of cubic forms is almost straightforward.

\begin{proposition}\label{equi:3}
Let $\form= a_1 x_1^3 + a_2 x_2^3 + a_3 x_3^3$ be a cubic diagonal form with $a_1,a_2,a_3\in\Z\setminus \{0\}$. 
For all $\b\in(-1,1]$ let
\[
\P_{\sform,\b} := \{\text{$p$ prime}:\ p \equiv 1 \mod* 3,\  p\not \in \Sigma_\sform,\  \text{and}\ \rhop\leq \beta\}.
\]
Then $\P_{\sform,\b}$ is nonempty, and for $T\to \infty$ we have
\begin{equation}\label{equi:H:3}
\#\P_{\sform,\b}\cap [1,T] = \d \Li(T)+ O(T e^{-\a\sqrt{\log T}})
\end{equation}
for some 
absolute constant $\a\great 0$ and with $\d =\frac 1 {2\pi} \op{arccos}(-\b)$.
\end{proposition}

\begin{proof}
Notice that $\P_{\sform,1}$ is just the set of all primes $p\equiv 1\pmod 3$ with $p\not\in\Sigma_\sform$, because $\rhop\leq 1$ is always satisfied. 
Now, recall from \cref{hecke} that the Jacobi sum symbol and the power residue symbols are Hecke characters of cyclotomic fields. 
For every $n\in\N$ we consider the unitary Hecke character
\[
H_{n}(\cdot) := \J_3(\cdot)^{n}
 \left(\frac {a_1a_2a_3}{\cdot}\right)_3^{-n}
\]
of the number field $K=\Q(e^{2\pi i /3})$. 
We have that $\m=\m_{a_1a_2a_3,3}\cap\m_{\J_3}\subseteq \O_K$ is a defining ideal of $H_n$ for every $n\in\N$ and the infinity type of $H_n$ is $\alpha\mapsto (\alpha/\abs{\alpha})^n$. 
Since the field $K$ has degree $d=2$, the size of $H_n$ satisfies $v_\xxi(H_n)\leq c_3 n$ for some $c_3>0$ independent of $n$.  
Moreover for $n\neq 0$ the character $H_n$ is nontrivial, because, since $\m$ is a lattice in $\C$, there exists $\alpha\in\Z[e^{{2\pi i}/{3}}]$ such that $\alpha\equiv 1\pmod \m$ and $\alpha^n\not\in\R$. 

The primes $\p\in\I_\m$ above a prime $\primo\equiv 1\pmod 3$ come in conjugate pairs, they satisfy $N\p=\primo$ and we have either $\p=\pi_{3,\primo}$ or $\wb{\p}=\pi_{3,\primo}$. Therefore from the definitions we have
\[
	H_n(\p)+H_n(\wb{\p}) = 2\Re((H_{\sform,\primo})^n).
\]
On the other hand the primes $\p\in\I_\m$ above a prime $\primo\equiv 2\pmod 3$ satisfy $N\p=\primo^2$, and so there are at most $\sqrt T$ of them satisfying $N\p\leq T$, for every given $T\great 0$. 
By these remarks, and by \cref{pnt} applied to $H_n(\cdot)$ we get, for every $T\geq 2$  and every positive integer $n\leq c_3^{-1} \exp(\sqrt {\log T} )$: 
\begin{equation}\label{equi:H:3:1}
\left|
\sum_{p \in\P_{\sform,1}\cap[1,T]}
\Re( (H_{\sform,p})^n)
\right|
        \leq   c_1\, T\, e^{-c_2\sqrt{\log T}} + \sqrt T,
\end{equation}
for some absolute 
 constants $c_1, c_2\great 0$. 
We observe that $\Hp$ belongs to the unit circle for all $p\in\P_{\sform,1}$ and that $\#\P_{\sform,1}\cap[1,T] = \frac 1 2 \Li(T)+ O(T e^{-c_{3}\sqrt{\log T}})$ for some effective absolute constant $c_3\great 0$, by the prime number theorem on arithmetic progressions. 
Now by \cref{ET} applied with $\phi_1=\op{arccos}(\beta)$, $\phi_1=\pi$ and $N=\lfloor  c_3^{-1} e^{\sqrt {\log T}}\rfloor$ we get the asymptotics displayed in \eqref{equi:H:3}, for any $\alpha\less \min\{1,c_2,c_3\}$.
\end{proof}

\subsection{Equidistribution of $\Hpp$}\label{sec:equidistribution:4}

For a biquadratic form $\form$ we need  that $\Hpp$ equidistributes when $q\to\infty$ ranges in the set $q\in\P_{\sform,\u}$ that we defined in \cref{sec:kummer}, for a fixed $\u\in\mu_4^4\times\{\pm 1\}$. 
To detect those primes among the primes in $\Pquattro$ 
(and so to handle sums indexed by them) we use character sums, as follows. 
We define the auxiliary polynomial $\pol(x):=1+x+x^2+x^3$ and the auxiliary sum
\begin{equation}\label{eq:detect}
\Saux(\u,\mbf v)
	=     \sum_{\k\in(\Z/4\Z)^5}  C(\u,\k) \prod_{i=1}^5 v_i^{k_i}   
\end{equation}
for all $\u,\mbf v\in\mu_4^4\times\{\pm 1\}$, where $C(\u,\k)= {2^{1-\epsilon_\u}} {4^{-5}}\Re(u_1^{k_1}u_2^{k_2}u_3^{k_3}u_4^{k_4}u_5^{k_5})$, $\epsilon_\u=1$ if $\u=\wb\u$ and $\epsilon_\u=0$ otherwise.  
Observe that 
\[
\Saux(\u,\mbf v)=   \frac {2^{-\epsilon_\u}} {4^5}\left(\prod_{i=1}^5\pol(u_i v_i) + \prod_{i=1}^5\pol(u_i^{-1} v_i)\right), 
\]  
from which we see that $\Saux(\u,\mbf v)=1$ if $\mbf v\in\{\u,\wb\u\}$ and $\Saux(\u,\mbf v)=0$ otherwise. 
Therefore, for $q\in\Pquattro$ we have $q\in \P_{\sform,\u}$ if and only if $\Saux(\u,\chipp(\ua,-1))=1$, where $\P_{\sform,\u}$ and $\chipp(\ua,-1)$ are as in \cref{sec:kummer}. 
In particular, for all $T\geq 1$, all $\u\in\mu_4^4\times\{\pm 1\}$ and every function $h: \Pquattro\to \C$ we have
\begin{equation}\label{eq:detect:2}
\sum_{q\in \P_{\sform,\u}\cap [1,T]} h(q) 
	= \sum_{q\in \Pquattro\cap [1,T]} \Saux(\u,\chipp(\ua,-1))\ h(q).
\end{equation}
This is a common technique in analytical number theory, see e.g. \cite[Lemma 4]{Kubilyus} for an application of this trick in a similar context. 

\begin{proposition}\label{equi:4}
Let $\form= a_1 x_1^4 +\ld + a_4 x_4^4$ be a biquadratic diagonal form with $a_1,\ld,a_4\in\Z\setminus \{0\}$. 
For all $\u\in\mu_4^4\times\{\pm 1\}$ and all $\b\in(-1,1]$ let
\[
\P_{\sform,\u,\b} := \{q\in\P_{\sform,\u}:\  \rhopp\leq \beta\}.
\]
If $\u$ is in the image of $\aphi_\sform$, 
then $\P_{\sform,\u,\b}$ is nonempty, and for $T\to \infty$ we have
\begin{equation}\label{equi:H:4}
\#\P_{\sform,\u,\b}\cap [1,T] = \d \Li(T)+ O(T e^{-\a\sqrt{\log T}})
\end{equation}
for some effective absolute constant $\a\great 0$ and with $\d \geq \frac 1 {1024\pi} \op{arccos}(-\b)$.
\end{proposition}

\begin{proof}
For all $n\in\N_+$ and all $\k=(k_1,\ldots,k_5)\in(\Z/4\Z)^5$ we define the unitary Hecke character
\[
H_{n,\k}(\cdot) := \J_{4}(\cdot)^{2n}
 \Big(\frac {a_1a_2a_3a_4}{\cdot}\Big)_{4}^{-n}
 \Big(\frac {a_1}{\cdot}\Big)_{4}^{k_1}
 \Big(\frac {a_2}{\cdot}\Big)_{4}^{k_2}
 \Big(\frac {a_3}{\cdot}\Big)_{4}^{k_3} 
 \Big(\frac {a_4}{\cdot}\Big)_{4}^{k_4}
 \Big(\frac {-1}{\cdot}\Big)_{4}^{k_5}
\]
of the number field $K=\Q(i)$. 
For every $n$ and $\k$ as above we have that 
\[
\m=\m_{a_1,4}\cap\m_{a_2,4}\cap\m_{a_3,4}\cap\m_{a_4,4}\cap\m_{-1,4}\cap\m_{\J_4}\subseteq\O_K
\]
 is a defining ideal of $H_{n,\k}$ and $\alpha\mapsto (\alpha/\abs{\alpha})^{2n}$ is its infinity type. 
Observe that $H_{n,\k}$ is nontrivial for $n\neq 0$ because, $\m$ being a lattice in $\C$, there exists $\alpha\in\Z[i]$ such that $\alpha\equiv 1\pmod \m$ and $\alpha^{2n}\not\in\R$. 
Moreover, the size of $H_{2n,\k}$ satisfies $v_\xxi(H_{n,\k})\leq c_3 n$ for some $c_3>0$ independent of $n$. 

The primes $\pp\in\I_\m$ with degree $\deg(\pp)\neq 1$ are precisely those above a prime $\pprimo\equiv 3\pmod 4$. 
These primes satisfy $N\pp=\pprimo^2$, and so there are at most $\sqrt T$ of them satisfying $N\pp\leq T$, for every given $T\great 0$. 
Therefore, by \cref{pnt} applied to $H_{n,\k}(\cdot)$ we get, for every $T\geq 2$, every positive integer $n\leq c_3^{-1} \exp(\sqrt{\log T})\in\N_+$ and every $\k\in(\Z/4\Z)^5$:
\begin{equation}\label{eq:hnk:1}
\left| {
\sum_{\substack{\pp\in\op{Spec}\O_K\cap\I_\m\\ N\pp\leq T,\, \deg(\pp)=1}}
  H_{n,\k}(\pp)
}\right|
  \leq c_1\, T\, e^{-c_2\sqrt {\log T}},
\end{equation}
for some 
constants $c_1,c_2\great 0$ independent of $n\in\N_+$ and $\k\in(\Z/4\Z)^5$.
The primes $\pp\in\I_\m$ with $\deg(\pp)$ come in conjugate pairs, they satisfy $N\p=\pprimo$ for some $\pprimo\equiv 1\pmod 4$ and we have either $\pp=\pi_{4,\pprimo}$ or $\wb{\pp}=\pi_{4,\pprimo}$. 
In particular, given such $\pp$ and $\mbf v = \chipp(\ua,-1)$ we have:
\begin{equation}\label{eq:detect:3}
H_{n,\k}(\pp) + H_{n,-\k}(\wb\pp) = 2\Re((\Hpp)^n) \prod_{i=1}^5 v_i^{k_i}
\end{equation}
 for all $n\in\N_+$ and all $\k\in(\Z/4\Z)^5$. 
Then \eqref{eq:detect}, \eqref{eq:detect:2} and \eqref{eq:detect:3} imply
\begin{equation}\label{eq:hnk:2}
\sum_{\pprimo \in \P_{\sform,\u}\cap [1,T]}  
    2\Re((\Hpp)^n) 
= 
\sum_{\k\in(\Z/4\Z)^5} C(\u,\k)
 \sum_{\substack{\pp\in\op{Spec}\O_K\cap\I_\m\\ N\pp\leq T,\,\deg(\pp)=1}}
  H_{n,\k}(\pp).
\end{equation}
for some real numbers $C(\u,\k)$ satisfying $C(\u,\k)=C(\u,-\k)$ and $\abs{C(\u,\k)}\leq \frac{1}{512}$. 
Finally, by \eqref{eq:hnk:1} and \eqref{eq:hnk:2} we deduce that 
\begin{equation}\label{eq:hnk:3}
\left|
\sum_{q \in \P_{\sform,\u}\cap [1,T]}
     \Re((\Hpp)^n) 
\right|
        \leq c_1\,  T\, e^{-c_2\sqrt{\log T}},
\end{equation}
for all $T\geq 2$ and all positive $n\leq c_3^{-1} \exp{\sqrt{\log T}}$.
We observe that $\Hpp$ belongs to the unit circle for all $p\in\P_{\sform,\u}\cap [1,T]$ and that $\#\P_{\sform,\u}\cap [1,T]$ is estimated in \cref{chebotarev}. 
Now by \cref{ET} applied with $\phi_1=\op{arccos}(\b)$, $\phi_1=\pi$ and $N=\lfloor c_3^{-1} \exp{\sqrt{\log T}}\rfloor$ we get the asymptotics displayed in \eqref{equi:H:3}, for any $\alpha\less \min\{1,c_2,\a'\}$. 
\end{proof}

%
%

\section{Detecting the existence of long gaps - the proof}
\label{sec:main}

\subsection{Congruences with few solutions}
\label{sec:main:cong}

For the remaining part of the article let $\s\in\{3,4\}$ and let $\form= a_1 x_1^\s +\ld + a_\s x_\s^\s$, with $a_1,\ld,a_\s\in\N_+$ be either a cubic diagonal form or a biquadratic diagonal form that is not exceptional according to \cref{def:exceptional}.

\begin{proposition}\label{main:modp}
Let $\s, \form$ be as above. 
Then we can choose a set $\P_\sform$ of prime numbers and effectively computable absolute constants $\ax,\bx,\dxx\great 0$ such that for all $p\in\P_\sform$ and all $m\in\Z$ we have 
\begin{align}
\sol_\sform(0,p) & \leq p^{\s-1}\left( 1  - \bx (p^{1-\frac \s 2} - p^{- \frac \s 2})\right),\label{eq:main:modp:0}\\
\sol_\sform(m,p) & \leq p^{\s-1}\left( 1  + (\s-1)^\s p^{\frac 1 2 -\frac \s 2}\right),\label{eq:main:modp:1}
\end{align}
and for $T\to \infty$ we have, for some $\dx\geq\dxx$:
\begin{equation}\label{eq:main:modp:P}
\#\P_\sform \cap[1,T]= \dx \Li(T)  + O ( T e^{-\ax \sqrt{\log T}}).
\end{equation}
\end{proposition}

\begin{proof}
If $\s=3$ the inequality \eqref{eq:main:modp:0} and the asymptotics \eqref{eq:main:modp:P} follow from \cref{erre:modp:3} and \cref{equi:3} by choosing any $\bx\in (0,2)$ and letting $\P_\sform:=\P_{\sform,-\bx/2}$. 
If $\s=4$ we may choose any $\bx\in (0,1)$, and let $\u$ be any element in the image of $\aphi_\sform$ such that $\proj(\u)$ satisfies \eqref{eq:c:e:0}. 
Then \eqref{eq:main:modp:0} and \eqref{eq:main:modp:P} follow from \cref{erre:modp:4}, \cref{chebotarev} and \cref{equi:4} with $\P_\sform:=\P_{\sform,\u,-\b/2}$. 
For both $\s\in\{3,4\}$ and for the same choice of $\P_{\sform}$, \eqref{eq:main:modp:1} follows from \eqref{eq:main:modp:0} when $p\divides m$ and it follows from \cref{erre:1} otherwise.

\end{proof}

From \cref{main:modp} we deduce that a biquadratic diagonal form satisfying $\sol_\sform(0,p)\geq p^3$ for all but finitely many primes $p$ must be exceptional in the sense of \cref{def:exceptional}. 
This observation, together with the arguments of \cref{sec:exceptional}, completes the proof of  \cref{thm:exceptional:iff}. 
For non-exceptional diagonal forms, \cref{main:modp} implies that the ratio $\sol_\sform(m,M)/M^{\s-1}$ can be made strictly less than 1 for suitable $m$ and $M=p$ prime. 
In the next proposition we make this ratio arbitrarily small by using products of primes.

\begin{proposition}\label{main:modM}
Let $\s,\form,\bx,\P_\sform$ be as in \cref{main:modp}. 
Let $\P_1\subseteq\P_2\subset \P_\sform$ with $\#\P_2\less\infty$ and let $m\in\Z$ with $m\equiv 0 \pmod p$ for all $p \in\P_1$. 
Then we have $\sol_\sform(m,M)\leq \e M^{s-1}$ for  $M:=\prod\limits_{p\in\P_2} p$ and all $\e\great 0$ that satisfy
\begin{equation}\label{eq:modM:e}
\log \e 
\geq 
-
\sum_{p\in\P_1}   \bx (p^{1 - \frac \s 2} - p^{- \frac \s 2})
+
\sum_{p\in\P_2\setminus \P_1}   (\s-1)^\s   p^{\frac 1 2 - \frac \s 2}.
\end{equation}
\end{proposition}

\begin{proof}
By \cref{erre:multiplicative} we have that 
\[
\frac{\sol_\sform(m,M)}{M^{\s-1}}= 
\prod_{p\in\P_1}\frac{\sol_\sform(0,p)}{p^{\s-1}}
\prod_{p\in\P_2\setminus\P_1}\frac{\sol_\sform(m,p)}{p^{\s-1}}.
\]
By \cref{main:modp} and the inequality $\log (1+x)\leq x$, valid for all $x\great -1$, we have that 
\[
\log\left(\prod_{p\in\P_1}\frac{\sol_\sform(0,p)}{p^{\s-1}} \right)	\leq \bx\sum_{p\in\P_1} (-p^{1 - \frac \s 2}+p^{-\frac \s 2}),
\]
and
\[
\log\left(\prod_{p\in\P_2\setminus\P_1}\frac{\sol_\sform(m,p)}{p^{\s-1}} \right)	\leq (\s-1)^\s \sum_{p\in\P_2\setminus\P_1} p^{\frac 1 2-\frac s 2},
\]
so the proposition follows.
\end{proof}

\subsection{Low density along arithmetic progressions}
\label{sec:main:ap}

Let $\s,\form$ be as in \cref{sec:main:cong}. 
\begin{definition}\label{def:erre}
For $n\in\N$ we define $\erre_\sform(n):=\#\Erre_\sform(n)$, where
\[
\Erre_\sform(n) :=\{\x\in\N^\s :\ \form = n\}.
\]
\end{definition}
In other words, $\erre_\sform(n)$ counts the number of representations of $n$ via the form $\form$. 
Then the image $\Set$ of $\form$ can be described as
\[
\Set:=\{n\in\N:\ \erre_\sform(n)\neq 0\}.
\]
The relative density of $\Set$ along an arithmetic progression of the form $m+M\N$ is related to $\sol_\sform(m,M)$. 
A trivial inequality relating the two is sufficient for our purpose. 

\begin{proposition}\label{ap:density}
Let $\s,\form,\Set$ be as above. 
Let $L,M,m\in\N_+$ with $m\less M$. 
Then
\begin{equation}
\label{eq:ap:density}
 \#\left(  \Set \cap (m+M\N)\cap [0,L^\s M^\s) \right) 	\leq 	\sol_\sform(m,M) L^\s.
\end{equation}
\end{proposition}

\begin{proof}
We consider the map
\begin{center}
\begin{tabular}{>{$}r<{$}>{$}c<{$}>{$}c<{$}>{$}r<{$}>{$}c<{$}}
\phi: \ & \displaystyle \bigcup_{k=1}^{L^\s M^{\s-1}} 	& \Erre_\sform(m+(k-1)M) & \too 	& \Sol_\sform(m,M)\\
	& 								& (x_1,\ld,x_{\s})		& \longmapsto & (x_1\bmod M,\ld,x_{\s}\bmod M)
\end{tabular}
\end{center}

and for every $k\leq L^\s M^{\s-1}$ we notice that $m+(k-1)M\less L^\s M^\s$. This implies that for every $\x\in\Erre_\sform(m+(k-1)M)$ and all $j\in\{1,\ld,\s\}$ we have $0\leq x_j\less LM$.
For every residue class $\wb x$ modulo $M$ there are only $L$ integers $x$ satisfying $x\equiv \wb x\pmod M$ and $0\leq x\less LM$.
We deduce that every element in the image of $\phi$ can have at most $L^\s$ preimages. 
Therefore
\[
\sum_{k=1}^{L^\s M^{\s-1}} \erre_\sform(m+(k-1)M) \leq L^\s \,\sol_\sform(m,M).
\]
Since 
\[
\#\left(  \Set \cap (m+M\N)\cap [0,L^\s M^\s)\right)  \leq \sum_{k=1}^{L^\s M^{\s-1}} \erre_\sform(m+(k-1)M),
\]
the proposition follows.
\end{proof}

Recall that we are interested in intervals contained in $\N\setminus\Set$, so next we consider arithmetic progressions of intervals with fixed length. The union of these intervals in arithmetic progression forms a ``rectangle'' of integers $\{m+(h-1)M + k:\ h\leq H, k\leq K\}$. A set of this form is sometimes known as a Maier matrix.

\begin{proposition}\label{ap:interval}
Let $\s,\form,\Set$ be as above, let $L,M,m,K\in\N_+$ with $m+K\less M$ and let
$\A=(m+\N M)\cap [0,L^\s M^\s)$ be a truncated arithmetic progression. 
Now let
\[
\B:=\{a\in\A:\  \Set\cap(a+[1,K])\neq \emptyset\}
\] 
and suppose that 
\begin{equation}\label{eq:interval}
\sol_\sform(m+1,M) + \ld + \sol_\sform(m+K,M) \leq \frac 1 2 M^{\s-1}.
\end{equation}
Then $\#\A = L^\s M^{\s-1}$ and $\#\B\leq \frac 1 2 L^\s M^{\s-1}$.
\end{proposition}

\begin{proof}
Since $m\less M$, the inequality $\#\A = L^\s M^{\s-1}$ is clear. 
To estimate $\#\B$, first notice that
\begin{equation}\label{eq:ap:interval}
\#\B \leq \sum_{h=1}^{L^\s M^{\s-1}} \#\left(\Set \cap(m+(h-1)M +  [1,K])\right)
\end{equation}
because each element of $\B$ contributes at least 1 to the sum in the right hand side of \eqref{eq:ap:interval}. 
Now observe that this sum is equal to
\[
 \sum_{i=1}^K  \#\left(\Set \cap(m+ i + M\N)\cap[0,L^\s M^{\s})\right).
\]
because $m+K\less M$. By \cref{ap:density} we deduce that 
\[
\#\B \leq \sum_{i=1}^K L^\s\,\sol_\sform(m+i,M) \leq \frac 1 2 L^\s  M^{\s-1}.
\]
\end{proof}

We remark that $\{a+1,\ldots,a+K\}$ is a gap in the values of $\form$ for any $a\in\A\setminus\B$ as in \cref{ap:interval}. 
In particular, the existence of such gaps follows from an inequality of the form \eqref{eq:interval}.

\subsection{Choice of parameters}
\label{sec:main:parameters}

In order to fulfil \eqref{eq:interval}, we will use the upper bounds on the summands $\erre_\sform(m+i,M)$ coming from \cref{main:modM} and from a suitable choice of sets $\P_1\subseteq\P_2\subset\P_\sform$. 
We will set $\P_2=\P_\sform\cap [1,T]$, for some $T$ large enough, and the next lemma is about finding the appropriate values of $T$.

\begin{definition}\label{def:tau}
For all $\gamma\in\R_+$ and $K\in \N_+$ we set
\[
\begin{aligned}
\tau_3(\gamma,K) &:= \gamma K^2 (\log K)^4;\\
\tau_4(\gamma,K) &:= \exp (\exp (\gamma K \log K )).
\end{aligned}
\]
\end{definition}

\begin{lemma}\label{tau}
Let $\s,\form,\P_\sform,\bx$ be as in \cref{main:modp}. 
Then there is $\gx\geq 1$ such that for all $K\in\N$ with $K\geq 2$, and all $T\geq \tau_{\s}(\gx,K)$, we have
\begin{equation}\label{eq:tau}
\log\left(\frac 1 {2K}\right) 
\geq 
\bx 
-
\frac 1 K \sum_{p\in \P_\sform\cap[1,T]}
          \bx(p^{1-\frac\s 2} - p^{-\frac \s 2})
+	
\sum_{p\in\P_\sform\cap[1,T]} 
       (\s-1)^\s   p^{\frac 1 2 - \frac \s 2}.
\end{equation}
In particular, for this choice of $T$ the set $\P_\sform\cap [1,T]$ is nonempty.
\end{lemma}

Please compare \eqref{eq:tau} with \eqref{eq:modM:e} and notice the extra multiplicative factor $\frac 1 K$ in front of the first sum.

\begin{proof}
Consider first the case $\s=4$. 
Let $\gamma\geq 1$ and $T\geq \tau_4(\gamma,K)$. 
By \cref{main:modp} and \cref{asymptotic} we have that:
\begin{alignat*}{3}
\text{(a)}\quad	
	&	
	&& \sum_{p\in\P_\sform\cap[1,T]} (\s-1)^\s  p^{\frac 1 2 - \frac \s 2} 	
	&& \leq 	C_1;  \\
\text{(b)}\quad	
	& \frac 1  K 
	&& \sum_{p\in \P_\sform\cap[1,T]} \bx(p^{1-\frac\s 2} - p^{-\frac \s 2})	
	&&  \geq \gamma \bx\dx \log K - {C_2};
\end{alignat*}
for some constants $C_1,C_2\great 0$ independent of $K$. 
Then  \eqref{eq:tau} holds if $\gamma\geq \gx$ for some $\gx$ that can be chosen independently of $K\geq 2$.  
Now we consider the case $\s=3$. 
Let $\gamma\geq 1$ and $T\geq \tau_3(\gamma,K)$. 
From \cref{asymptotic,main:modp} we have that:
\begin{alignat*}{3}
\text{(a)}\quad	
	&	
	&& \sum_{p\in\P_\sform\cap[1,T]} (\s-1)^\s  p^{\frac 1 2 - \frac \s 2} 	
	&& \leq 	C_3 \log\log \max\{\gamma,K\};  \\
\text{(b)}\quad	
	& \frac 1  K 
	&& \sum_{p\in \P_\sform\cap[1,T]} \bx(p^{1-\frac\s 2} - p^{-\frac \s 2})	
	&&  \geq     C_4 \frac {\sqrt\gamma(\log K)^2} {\log \max\{\gamma,K\}};
\end{alignat*}
for some constants $C_3,C_4\great 0$ independent of $K$. 
Again, it is easy to see that  \eqref{eq:tau} holds if $\gamma\geq \gx$ for some $\gx$ that can be chosen independently of $K$. 
Finally, we observe that $\bx+\log(2K)\great 0$, so \eqref{eq:tau} doesn't hold if $\P_\sform\cap[1,T]=\emptyset$.
\end{proof}

\subsection{Conclusion}\label{sec:main:main}
Let $\s,\form$ be as in \cref{sec:main:cong}. 
Given $N,K\in\N_+$, we define
\[
\Gap_\sform (N,K) :=\{n\in\N: \ n\less N \text{ and } \Set \cap (n+[1,K])=\emptyset\}.
\]
We aim to show that for every $K$ there is $N\in\N_+$ such that $\Gap_\sform(N,K)$ is nonempty. 

\begin{theorem}\label{thm:main}
Let $\s,\form$ be as in \cref{sec:main:cong}. 
Then for all $K\geq 2$ there is a constant $\cxx\great 0$ such that for all $N\geq  e^{\s \cxx}$ we have 
\[
\#\Gap_\sform(N,K)\geq \frac {e^{-\cxx}} {32}  N.
\]
Moreover we can choose $\cxx=(\dx + o(1)) \tau_{\s}(\gx,K)$ as $K\to\infty$, where $\dx$ is as in \cref{main:modp}, and $\gx$ is as in \cref{tau}.
\end{theorem}

\begin{proof}
Fix $K\geq 2$ and let $T \geq  \tau_{\s}(\gx,K)$.
By \cref{tau} we have that $\P_\sform\cap[1,T]\neq\emptyset$, so let
\[
 M:=\prod_{p\in\P_\sform\cap[1,T]} p,
\]
let $N\in\N$ with $N\geq M^\s$, and let $L:=\lfloor \sqrt[s] N / M \rfloor$.
Since $p^{1-\frac s 2}-p^{-\frac\s 2}\less 1$ for all $p\geq 1$, we can easily construct a partition
\[
 \P_\sform\cap [1,T] = \P_\sform^{(1)}\sqcup\dots \sqcup \P_\sform^{(K)}
\]
such that for all $i\in\{1,\ld,K\}$ we have 
\begin{equation}\label{eq:main:PK}
\sum_{p\in\P_\sform^{(i)}}( p^{1-\frac \s 2} - p^{-\frac \s 2})	 \geq 	-1+ \frac 1 K \sum_{p\in\P_\sform\cap[1,T]}( p^{1-\frac \s 2} - p^{-\frac \s 2}).
\end{equation}
By the Chinese Remainder Theorem there is some $m\in\N$ with $m\less M$ that satisfies $m\equiv -i \pmod p$ for all $i\in\{1,\ld,K\}$ and all $p\in \P_\sform^{(i)}$. 
By \eqref{eq:main:PK}, \cref{tau} and \cref{main:modM} with $\e=\frac 1 {2K}$ we deduce that
\[
\sol_\sform(m+i,M) \leq \frac 1 {2K} M^{\s-1}
\]
for all $i\in\{1,\ld,K\}$. 
Then \cref{ap:interval} implies that 
\[
\#\Gap_\sform (L^\s M^\s,K) 	\geq  	\frac 1 2 L^\s M^{\s-1} = \frac {(L+1)^\s M^\s }{2M}\left(\frac L {L+1}\right)^\s \geq \frac N{2M} \left(\frac 1 2 \right)^\s \geq  \frac N{32M}.	
\]
By \cref{asymptotic} we have that $\log M = T (\dx+o(1))$ as $T\to\infty$. 
Since  $\tau_{\s}(\gx,K)\to\infty$ as $K\to \infty$, and since $\Gap_\sform (L^\s M^\s,K)\subseteq \Gap_\sform (N,K)$, the theorem follows.
\end{proof}

We remark that, despite the appearances, in general a larger value of $\dx$ corresponds to a smaller value of $\cxx$ in \cref{thm:main}. 
As a corollary of \cref{thm:main} we get the theorems stated in the Introduction.

\begin{proof}[Proof of \cref{thm:main:intro:3}.]
Let $\s=3$, let $\form$ be as in \cref{thm:main:intro:3} and let $N,K\in\N$ with $K\geq 2$. 
From \cref{thm:main} (applied to estimate $\#\Gap_\sform (N-K,K)$) it is possible to compute some constant $\gamma\great 0$, independent of $N$ and $K$, such that whenever the inequality
\begin{equation}\label{eq:main:intro:3}
N\geq \exp(\gamma K^2(\log K)^4)
\end{equation}
holds, there is a gap of length $K$ in the values of $\form$ less than $N$. 
When $N\geq e^e$ we can write $K=\kappa\frac{\sqrt{\log N}}{(\log\log N)^2}$ for some $\kappa\great 0$. 
If $\kappa\leq 1$ we have $\log K\leq\frac 1 2 \log \log N$, so  \eqref{eq:main:intro:3} holds if moreover
\[
N\geq \exp\left(\gamma\kappa^2 \frac {1}{2^4}\log N\right),
\]
which is satisfied when $\kappa\leq\kx:=\min\{1,4/\sqrt\gamma)$.
\end{proof}

\begin{proof}[Proof of \cref{thm:main:intro:4}.]
Let $\s=4$, let $\form$ be as in \cref{thm:main:intro:4} and let $N,K\in\N$ with $K\geq 2$. 
As in the previous case, we deduce from \cref{thm:main} that there is some constant $\gamma\great 0$ independent of $N$ and $K$  such that appropriate gaps of length $K$ esist when the inequality
\begin{equation}\label{eq:main:intro:4}
N\geq \exp(\exp(\exp(\gamma K\log K)))
\end{equation}
holds. 
When $N\geq e^{e^{e^{e}}}$ we can write $K=\kappa\frac{\log\log\log N}{\log\log\log\log N}$ for some $\kappa\great 0$. 
If $\kappa\leq 1$ we have $\log K\leq \log\log\log\log N$, so \eqref{eq:main:intro:4} holds if moreover
\[
N\geq \exp(\exp(\exp(\gamma\kappa\log\log\log N))),
\]
which is satisfied when $\kappa\leq\kx=\min\{1,1/\gamma)$.
\end{proof}

\begin{remark}\label{rmk:easy}
For some diagonal forms a more elementary proof can be given, i.e. not involving Hecke characters and Chebotarev's theorem for abelian extensions. 
For example for the biquadratic diagonal form $\form= x_1^4+x_2^4+x_3^4+x_4^4$ we notice that $\Kpp=-5$ for all $q\equiv 5\pmod 8$. 
Since $\rhopp\leq 2$, we see that \cref{main:modp} holds with $\bx=3$ and $\P_\sform=\{q\text{ prime}:\  q\equiv 5\pmod 8\}$, even without referring to the equidistribution of $\Hpp$.

We can avoid the reference to an equidistribution result for cubic forms as well. 
For example for $\form= x_1^3+x_2^3+x_3^3$  we can prove that if $p\equiv 1 \pmod 3$ and $m$ is a nonzero noncubic residue class modulo $p$ (i.e. $\chip(m)\not\in\{0,1\}$), then $\sol_\sform(m,p)\leq p^2 -3p +2\sqrt p$. 
This is enough to imply the existence of unbounded gaps, though with smaller size compared to \cref{thm:main:intro:3}.
With this alternative approach, it helps to observe that for all primes $p$ large enough we can find $K$ consecutive residue classes modulo $p$ at which $\chip$ assumes any given value, see \cite{MO}.
\end{remark}

\section*{Acknowledgements}

I would like to thank my supervisor Damien Roy for his encouragement and for his many comments on this work. 
Among the many people to whom I had the pleasure to speak about this project, I am specially grateful to Simon Rydin Myerson and Marc Hindry for their interesting remarks. 
I also thank Martin Rivard-Cooke for having introduced me to the problem of gaps for $\form= x_1^3+x_2^3+x_3^3$ and Daniel Fiorilli for his comments on the content of the paper.  
For their help in finding references, I thank Daniel Fiorilli, Gerry Myerson and the user EFinat-S from Mathoverflow. 
I thank Kam Hung Yau for spotting some typos in a previous version of the paper. 
I thank an anonymous referee for valuable suggestions, especially concerning the introduction. 
Finally, I thank Francesco Veneziano for discussing with me the problem of gaps in the case of degree two: 
the strategy followed in this article was designed as an attempt to generalize our computations to higher degree.
This work was supported in part by a full International Scholarship from the Faculty of Graduate
and Postdoctoral Studies of the University of Ottawa and by NSERC. 

\bibliographystyle{abbrv}
\bibliography{biblio_gaps_diagonal}

\end{document}